\documentclass[a4paper,12pt,reqno]{amsart} 

\makeatletter
\tagsleft@false 
\makeatother

\usepackage{comment}

\usepackage{tikz}
\usepackage{pgfplots}
\pgfplotsset{compat=1.18}
\usepackage{longtable}
\usepackage{amsmath}
\usepackage{amssymb}
\usepackage{amsthm}
\usepackage{mathrsfs}
\usepackage{ifthen}
\usepackage{graphicx}
\usepackage[T1]{fontenc}

\numberwithin{equation}{section}

\newtheorem*{theorem*}{Theorem}

\newtheorem{thm}{Theorem}[section]
\newtheorem{lem}{Lemma}[section]

\newtheorem*{cor*}{Corollary}

\theoremstyle{definition}

\newtheorem{prob}{Problem}[section]

\newtheorem{cor}{Corollary}[section]

\newcounter{minutes}
\divide\time by 60
\newcounter{hours}
\multiply\time by 60
\addtocounter{minutes}{-\time}

\newcounter{own}
\def\theown{\thesection.\arabic{own}}

\newenvironment{pf}[1][]{%
	\vskip 3mm
	\noindent
	\ifthenelse{\equal{#1}{}}%
	{{\slshape Proof. }}%
	{{\slshape #1.} }%
}%
{\qed\bigskip}

\newcounter{alphabet}

\def\be{\begin{equation}}
	\def\ee{\end{equation}}

\newcommand{\bee}{\begin{enumerate}}
	\newcommand{\eee}{\end{enumerate}}

\newcommand{\blem}{\begin{lem}}
	\newcommand{\elem}{\end{lem}}
\newcommand{\bthm}{\begin{thm}}
	\newcommand{\ethm}{\end{thm}}
\newcommand{\bcor}{\begin{cor}}
	\newcommand{\ecor}{\end{cor}}
\newcommand{\eeg}{\end{examp}}

\newcommand{\eegs}{\end{examples}}
\newcommand{\edefe}{\end{defin}}
\newcommand{\bprob}{\begin{prob}}
\newcommand{\eprob}{\end{prob}}
\newcommand{\bei}{\begin{itemize}}
\newcommand{\eei}{\end{itemize}}

\allowdisplaybreaks

\begin{document}

\title{{The second and third Hankel determinants for starlike MA--Minda subclass associated to quadratic polynomials}}

\author{Vasudevarao Allu}
\address{Vasudevarao Allu,
Department of Mathematics,
School of Basic Sciences,
Indian Institute of Technology Bhubaneswar,
Bhubaneswar-752050, Odisha, India.}
\email{avrao@iitbbs.ac.in}

\author{Shobhit Kumar}
\address{Shobhit Kumar,
Department of Mathematics,
School of Basic Sciences,
Indian Institute of Technology Bhubaneswar,
Bhubaneswar-752050, Odisha, India.}
\email{a21ma09007@iitbbs.ac.in}

\subjclass[2020]{30C45, 30C50, 41A10}

\keywords{ Analytic functions, Subordination,  Hankel determinants, Starlike function, Bernstein Polynomial. }

\def\thefootnote{}
\footnotetext{
{ }
}
\makeatletter\def\thefootnote{\@arabic\c@footnote}\makeatother

\begin{abstract}
	Let  $\mathcal{A}$ denote  the class of analytic functions such that $f(0)=0$ and $f'(0)=1$ in the unit disk $\mathbb{D}:=\{z \in \mathbb{C}: |z|<1\}.$
	In this paper, we discuss the properties of a starlike subclass and compute its second and third Hankel determinants; where the class is defined as  $\mathcal{S}^*(\varphi):=\{f\in\mathcal{A}:{zf'(z)}/{f(z)}\prec \varphi(z):=1+z+{m}/{n}\,\,  z^2,\text{ such that } 2m \le n, \text{ where } m,n\in\mathbb{N}\}.$    Furthermore, we show that the bounds are sharp by determining the extremal functions for the Hankel determinants.
\end{abstract}
\maketitle
\pagestyle{myheadings}
\markboth{V. Allu and S. Kumar}{ The second and third Hankel determinants for subclass of starlike fucntions  }
\setcounter{page}{1}

\section{Introduction}
Let $\mathcal{A}$ denote the class of normalized analytic functions in the open unit disk $\mathbb{D}:=\{z\in\mathbb{C}:|z|<1\}$, satisfying $f(0)=0$ and $f'(0)=1$.  We refer to Goodman \cite{Goodman1983} for basic background on these classical families.
Every function $f\in\mathcal{A}$ admits the Taylor series expansion
\[
f(z)=z+\sum_{n=2}^{\infty}a_n z^n.
\]
The subclass of $\mathcal{A}$ consisting of all univalent functions in $\mathbb{D}$ is denoted by $\mathcal{S}$. Thus,
$
\mathcal{S}:=\{f\in\mathcal{A}: f \text{ is univalent in } \mathbb{D}\}.
$
The class $\mathcal{S}$ plays a central role in geometric function theory. Its importance stems from the fact that univalent functions map the unit disk conformally onto simply connected domains in the complex plane. Two important subclasses of $\mathcal{S}$ are the classes of starlike and convex functions, which are defined in terms of the geometric properties of their image domains.\\[2mm]
A domain \(D\subset \mathbb{C}\) is said to be starlike with respect to a point \(z_0\in D\) if, for each \(z\in D\), the entire line segment joining \(z_0\) to \(z\) is contained in \(D\). That is,
\[
(1-t)z_0+t z \in D \qquad \text{for all } z\in D \text{ and all } t\in[0,1].
\]
In particular, if \(z_0=0\), then \(D\) is said to be starlike with respect to the origin, and the above condition reduces to
\[
tz\in D \qquad \text{for all } z\in D \text{ and all } t\in[0,1].
\]
Similarly, a domain \(D\subset \mathbb{C}\) is said to be convex if, for any two points \(z_1,z_2\in D\), the entire line segment joining \(z_1\) and \(z_2\) lies in \(D\). Equivalently,
\[
(1-t)z_1+t z_2 \in D \qquad \text{for all } z_1,z_2\in D \text{ and all } t\in[0,1].
\]
A function \(f\in\mathcal{A}\) is said to be convex in \(\mathbb{D}\) if and only if
\[
\operatorname{Re}\left(1+\frac{z f''(z)}{f'(z)}\right)>0
\qquad \text{for } z\in\mathbb{D},
\]
and \(f\) is said to be starlike in \(\mathbb{D}\) if and only if
\[
\operatorname{Re}\left(\frac{z f'(z)}{f(z)}\right)>0
\qquad \text{for } z\in\mathbb{D}.
\]
In order to unify and generalize many well-known subclasses of starlike and convex functions, Ma and Minda \cite{MaMinda1992} introduced a systematic approach based on subordination. Let $\varphi$ be analytic in the unit disk $\mathbb{D}$ with $\operatorname{Re}\varphi(z)>0$, normalized by $\varphi(0)=1$ and $\varphi'(0)>0$, and further assume that $\varphi(\mathbb{D})$ is symmetric with respect to the real axis and starlike with respect to $1$. Then a function $f\in\mathcal{A}$ is said to belong to the Ma--Minda convex subclass $\mathcal{C}(\varphi)$ if, and only if,
\[
1+\frac{z f''(z)}{f'(z)}\prec \varphi(z)
\qquad \text{for}\quad z\in\mathbb{D},
\]
and to the Ma--Minda starlike subclass $\mathcal{S}^*(\varphi)$ if, and only if,
\[
\frac{z f'(z)}{f(z)}\prec \varphi(z)
\qquad \text{for}\quad z\in\mathbb{D}.
\]
The Ma--Minda framework contains many important subclasses as special cases and provides a flexible setting for the study of coefficient problems, radius problems, and extremal questions (see \cite{Janowski1971, MaMinda1992}).\\[2mm]
In this paper, we consider the function
$
\varphi(z)=1+z+{m}/{n}\, \, z^2$,
where  $m,n\in\mathbb{N},
$
and study the associated Ma--Minda starlike subclass
\[
\mathcal{S}^*(\varphi):=\left\{f\in\mathcal{A}:\frac{z f'(z)}{f(z)}\prec \varphi(z)\right\}.
\]
We first verify that, under a suitable condition on $m$ and $n$, the function $\varphi$ satisfies the standard Ma--Minda assumptions. We then derive sharp bounds for the second and third Hankel determinants for functions in $\mathcal{S}^*(\varphi)$ and identify the corresponding extremal functions.

\section{Certain properties of the  $\varphi(z)$}
First, we prove that the function $\varphi(z)$ satisfies the assumptions  of  Ma-Minda function. It is easy to see that $\varphi(0)=1$ and $\varphi'(0)=1,$ and that the function $\varphi(z)$ is symmetric with respect to the real axis.
Next, we prove that $\varphi(z)$ is univalent in $\mathbb{D}.$
Taking, 
$
a={m}/{n}>0,
$
then we may rewrite $\varphi$ as 
$
\varphi(z)=1+z+az^2.
$
Next, we determine the relation between $m$ and $n$ for which \(\varphi\) is univalent in $\mathbb{D}$, let \(z_1,z_2\in\mathbb D\). We compute
\[
\varphi(z_1)-\varphi(z_2)
=(z_1-z_2)+a(z_1^2-z_2^2)
=(z_1-z_2)\bigl(1+a(z_1+z_2)\bigr).
\]
Therefore, if \(z_1\ne z_2\), and  we have 
$
\varphi(z_1)=\varphi(z_2),
$
then, it follows that 
\[
1+a(z_1+z_2)=0,
\]
that is,
\[
z_1+z_2=-\frac1a.
\]
Thus \(\varphi\) fails to be univalent precisely when there exist distinct points
\(z_1,z_2\in\mathbb D\) such that
$
z_1+z_2=-1/a.
$
Now, for \(z_1,z_2\in\mathbb D\), we have
$
|z_1+z_2|\le |z_1|+|z_2|<2.
$
Hence, if such \(z_1,z_2\) exist, then necessarily
$
|-1/a|<2,
$
or equivalently,
$
1/a<2,
$
which implies
$
a>1/2.
$
Therefore,  if \(a\le 1/2\), then no such pair \(z_1,z_2\in\mathbb D\) can exist, and thus 
\(\varphi\) is univalent in $\mathbb{D}$ for this case.\\
Conversely, suppose \(a>1/2\). Then
$1/a<2,
$
and taking
$
c=-{1}/{2a}.
$
Then \(|c|<1\). Choose \(\varepsilon>0\) small enough so that
\[
z_1=c+\varepsilon,
\qquad
z_2=c-\varepsilon
\]
both belong to $\mathbb{D}$. Clearly \(z_1\ne z_2\), and
\[
z_1+z_2=2c=-\frac1a.
\]
Hence
\[
1+a(z_1+z_2)=1+a\left(-\frac1a\right)=0,
\]
and therefore
\[
\varphi(z_1)=\varphi(z_2).
\]
So \(\varphi\) is not univalent in $\mathbb{D}$.
We conclude that \(\varphi\) is univalent in $\mathbb{D}$ if, and only if,
$
a\le 1/2.
$
Since \(a={m}/{n}\), this is equivalent to
$
{m}/{n}\le 1/2,
$
that is,
$
2m\le n.
$ \\[2mm]
Next, we claim that the function $\varphi(z)$ is starlike with respect to the point $\varphi(0)=1.$ if, and only if, $2 m \le n.$
Setting, 
$
a={m}/{n}>0,
$
 we may write $\varphi(z)$ as
$
\varphi(z)=1+z+az^2.
$
The function \(\varphi\) is starlike with respect to \(\varphi(0)=1\) if, and only if,
\[
g(z):=\varphi(z)-1=z+az^2
\]
is starlike with respect to the origin.
Now,
\[
g'(z)=1+2az,
\]
and therefore
\begin{align*}
\frac{z g'(z)}{g(z)}
=\frac{z(1+2az)}{z(1+az)}
=\frac{1+2az}{1+az}.
\end{align*}
By the analytic definition of starlike functions, $g(z)$ is starlike in $\mathbb{D}$ if, and only if,
\[
\operatorname{Re}\left(\frac{z g'(z)}{g(z)}\right)>0, \quad
\text{ for } z\in\mathbb D,
\]
that is,
\[
\operatorname{Re}\left(\frac{1+2az}{1+az}\right)>0
\quad
\text{ for } z\in\mathbb D.
\]
First, we  suppose that
$
a\le 1/2,
$
\, let $w=az$. Then \(|w|<a\le 1/2\), and
\[
\frac{1+2az}{1+az}
=\frac{1+2w}{1+w}
=1+\frac{w}{1+w}.
\]
Also,
\[
\left|\frac{w}{1+w}\right|
\le \frac{|w|}{1-|w|}
<1.
\]
Hence
\[
\operatorname{Re}\left(1+\frac{w}{1+w}\right)>0,
\]
which gives
\[
\operatorname{Re}\left(\frac{1+2az}{1+az}\right)>0,
\qquad z\in\mathbb D.
\]
Thus $g(z)$ is starlike in $\mathbb{D}$, and so \(\varphi\) is starlike with respect to \(1\).
Conversely, suppose that
$
a>1/2.
$
Choose \(r\) such that
\[
\frac{1}{2a}<r<\min\left\{1,\frac{1}{a}\right\},
\]
and set \(z=-r\). Then \(z\in\mathbb D\), and
\[
\frac{1+2az}{1+az}
=\frac{1-2ar}{1-ar}.
\]
Since \(r>{1}/{2a}\), we have \(1-2ar<0\), and since \(r<1/a\), we have \(1-ar>0\). Therefore
\[
\frac{1-2ar}{1-ar}<0.
\]
Hence,
\[
\operatorname{Re}\left(\frac{1+2az}{1+az}\right)<0
\]
for this choice of $z$, thus $g$ is not starlike in $\mathbb{D}$. Consequently, \(\varphi\) is not starlike with respect to \(1\).
Therefore \(\varphi\) is starlike with respect to \(\varphi(0)=1\) if, and only if, 
$
a\le 1/2.
$ Therefore, the claim is true.\\[2mm]
Next, we prove that the $\operatorname{Re}\varphi(z)$ remains positive in the $\mathbb{D}.$
Again taking 
$
a={m}/{n}>0.
$
Then
$
\varphi(z)=1+z+az^2.
$
We want to determine for which \(a>0\) one has
\[
\operatorname{Re} \varphi(z)>0 \quad
\text{ for } z\in\mathbb D.
\]
Consider the harmonic function
$
u(z)=\operatorname{Re} \varphi(z).
$
Since $u$ is harmonic in $\mathbb{D}$ and continuous on \(\overline{\mathbb D}\), it is enough to study its boundary values on \(|z|=1\).\\
Let \(z=e^{it}\), where \(t\in\mathbb R\). Then
\[
u(e^{it})=\operatorname{Re}\, \bigl(1+e^{it}+ae^{2it}\bigr)
=1+\cos t+a\cos 2t.
\]
Since,
$
x=\cos t\in[-1,1],
$
we obtain
\[
u(e^{it})=1+x+a(2x^2-1)=2ax^2+x+1-a.
\]
Thus it suffices to determine when the quadratic polynomial
\[
q(x)=2ax^2+x+1-a
\]
is non-negative for all \(x\in[-1,1]\).
Now
\[
q'(x)=4ax+1.
\]
Hence the critical point is
$
x_0=-{1}/{4a}.
$
Let us discuss the following two cases.\\[2mm]
Case1:  If $0<a\le 1/4$, then
$
x_0=-{1}/{4a}\le -1.
$
Since \(q\) is convex, its minimum on \([-1,1]\) occurs at \(x=-1\). Therefore
\[
\min_{x\in[-1,1]} q(x)=q(-1)=2a-1+1-a=a>0.
\]
Hence
$
u(e^{it})>0,$
for $|z|=1$.\\[2mm]
Case2: Now suppose \(a>1/4\). Then
$
x_0=-{1}/{4a}\in(-1,0),
$
so the minimum of \(q\) on \([-1,1]\) occurs at \(x=x_0\). A direct computation gives
\[
q\!\left(-\frac{1}{4a}\right)
=2a\left(\frac{1}{16a^2}\right)-\frac{1}{4a}+1-a
=1-a-\frac{1}{8a}.
\]
Thus \(q(x)\ge 0\) on \([-1,1]\) if, and only if, 
\[
1-a-\frac{1}{8a}\ge 0.
\]
Since \(a>0\), this is equivalent to
$
8a-8a^2-1\ge 0,
$
The roots of the quadratic equation
$
8a^2-8a+1=0
$
are
$
a={(2\pm\sqrt2)}/{4}.
$
Therefore
$
8a^2-8a+1\le 0, 
$
thus, we have
\[
\frac{2-\sqrt2}{4}\le a\le \frac{2+\sqrt2}{4}.
\]
Since in the present case \(a>1/4\), only the upper bound is relevant, and we obtain
$
a\le {(2+\sqrt2)}/{4}.
$
Combining the two cases, we conclude that
\[
u(e^{it})\ge 0 \qquad \text{for} \quad |z|=1
\]
if, and only if, 
\[
0<a\le \frac{2+\sqrt2}{4}.
\]
Because $u$ is harmonic and not identically zero, the boundary condition \(u\ge 0\) implies
$
u(z)>0 
$
in $\mathbb{D}.$
Hence,
$
\operatorname{Re} \varphi(z)>0
$
if, and only if, 
$
a\le {(2+\sqrt2)}/{4}.
$
Since \(a={m}/{n}\), this is equivalent to
\[
\frac{m}{n}\le \frac{2+\sqrt2}{4}.
\]
Therefore we obtain that $\varphi(z)$ satisfies all the Ma-Minda properties given that $2m \le n,$ where $m, n\in \mathbb{N}.$ Hence, we use the convention $2 \, m \le n$ for the rest of this paper. 

\section{Sharp Hankel Determinants}
In 1966, Pommerenke \cite{Pommerenke1966} introduced the concept of the Hankel determinants for the class $\mathcal{S}.$ The expression of the $q$th Hankel determinant for a function $f \in \mathcal{A}$, whose coefficients are given as follows:
\[
H_q(n) =
\begin{vmatrix}
a_n & a_{n+1} & \cdots & a_{n+q-1} \\
a_{n+1} & a_{n+2} & \cdots & a_{n+q} \\
\vdots & \vdots & \ddots & \vdots \\
a_{n+q-1} & a_{n+q} & \cdots & a_{n+2q-2}
\end{vmatrix}
\quad \mbox{for } q, n \in \mathbb{N}.
\]
For some special choices of $n$ and $q$, the second Hankel determinant $H_{2}(2)$ is defined as 
\begin{equation}\label{eq:3.1}
H_2(2) = a_2 a_4 - a_3^{\,2},
\end{equation}
which is called the second Hankel determinant of order $2$. For our case $a_1 := 1$, therefore, the expression of the third order Hankel determinant is given by
\begin{align*}
H_3(1) =
\begin{vmatrix}
1 & a_2 & a_3 \\
a_2 & a_3 & a_4 \\
a_3 & a_4 & a_5
\end{vmatrix}
= a_3(a_2 a_4 - a_3^2) - a_4(a_4 - a_2 a_3) + a_5(a_3 - a_2^2).
\end{align*}
 Next, we introduce the Carath\'eodory class whose coefficients will help us determining the bounds Hankel determinants.
 Let $\mathcal{P}$ denote the Carath\'eodory class consisting of all analytic functions
 $p$ in the unit disk $\mathbb{D}$ satisfying
 $
 p(0)=1$
and $
 \operatorname{Re} p(z)>0$
 $z\in\mathbb{D}$.
 Thus, every function $p\in\mathcal{P}$ has the form
 \[
 p(z)=1+\sum_{n=1}^{\infty} p_n z^n.
 \]
The following lemma serves as a basic tool for establishing our main result, as it contains the well-known formulas for $p_2$, $p_3$, and $p_4$.
\begin{lem} {\label{lemma2}}\textnormal{{\cite{KwonLeckoSim2018}, \cite{LiberaZlotkiewicz1982}}}
Let $p \in \mathcal{P}$ be of the form $p(z) = 1 + \sum_{n=1}^{\infty} p_n z^n$. Then
\begin{align*}
2 p_2 &= p_1^2 + \gamma \bigl(4 - p_1^2\bigr), \\[2mm]
4 p_3 &= p_1^3 + 2\bigl(4 - p_1^2\bigr)p_1 \gamma
- \bigl(4 - p_1^2\bigr)p_1 \gamma^2
+ 2\bigl(4 - p_1^2\bigr)\bigl(1 - |\gamma|^2\bigr)\eta, \\[2mm]
8 p_4 &= p_1^4
+ \bigl(4 - p_1^2\bigr)\gamma\bigl(p_1^2(\gamma^2 - 3\gamma + 3) + 4\gamma\bigr)\\[1mm]
&\quad- 4\bigl(4 - p_1^2\bigr)\bigl(1 - |\gamma|^2\bigr)
\left(p_1(\gamma - 1)\eta + \overline{\gamma}\,\eta^2
- (1 - |\eta|^2)\rho\right)
\end{align*}
for some complex numbers $\gamma$, $\eta$, and $\rho$ such that $|\gamma| \leq 1$, $|\eta| \leq 1$, and $|\rho| \leq 1$.
\end{lem}
Next, we recall the following well-known result by Choi \textit{et al.} \cite{ChoiKimSugawa2007}
\begin{lem}\label{lemma3} {\cite{ChoiKimSugawa2007}}
Let $A,B,C\in\mathbb R$ and define
\begin{equation*}
Y(A,B,C):=\max_{z\in\overline{\mathbb D}}\Big(|A+Bz+Cz^{2}|+1-|z|^{2}\Big).
\end{equation*}
\begin{itemize}
\item[(i)] If $AC\ge 0$, then
\begin{align*}
	Y(A,B,C)=
	\begin{cases}
		|A|+|B|+|C|, & \text{if } |B|\ge 2\,(1-|C|),\\[6pt]
		1+|A|+\dfrac{B^{2}}{4\,(1-|C|)}, & \text{if } |B|<2\,(1-|C|).
	\end{cases}
\end{align*}
\item[(ii)] If $AC<0$, then
\begin{align*}
	Y(A,B,C)=
	\begin{cases}
		1-|A|+\dfrac{B^{2}}{4\,(1-|C|)}, & \text{if } -4AC\,(C^{2}-1)\le B^{2}\ \text{ and }\ |B|<2\,(1-|C|),\\[8pt]
		1+|A|+\dfrac{B^{2}}{4\,(1+|C|)}, & \text{if } B^{2}<\min\!\big\{\,4(1+|C|)^{2},\ -4AC\,(C^{2}-1)\big\},\\[8pt]
		R(A,B,C), & \text{otherwise,}
	\end{cases}
\end{align*}
where,
\begin{align*}
	R(A,B,C)=
	\begin{cases}
		|A|+|B|+|C|, & \text{if } |C|\big(|B|+4|A|\big)\le |AB|,\\[6pt]
		-|A|+|B|+|C|, & \text{if } |AB|\le |C|\big(|B|-4|A|\big),\\[10pt]
		\big(|A|+|C|\big)\sqrt{\,1-\dfrac{B^{2}}{4AC}\,}, & \text{otherwise.}
	\end{cases}
\end{align*}
\end{itemize}
\end{lem}
The next well-known lemma by Prokhrov and Szynal \cite{ProkhorovSzynal1981} will be crucial for our result.
\begin{lem}\label{lem:schwarz-param}
Let
$
w(z)=\sum_{n=1}^{\infty}c_n z^n
$
be a Schwarz function, that is, $w$ is analytic in $\mathbb D$,
$w(0)=0$, and $|w(z)|<1$ for all $z\in\mathbb D$.
If $c_1\ge 0$, then there exist complex numbers $\gamma,\eta,\rho$
with
$
|\gamma|\le 1,\, |\eta|\le 1,\, |\rho|\le 1,
$
such that
\begin{align*}
c_2 &= (1-c_1^2)\gamma,\\[2mm]
c_3 &= (1-c_1^2)\bigl(\eta(1-|\gamma|^2)-c_1\gamma^2\bigr),\\[2mm]
c_4 &= (1-c_1^2)\Bigl(
c_1^2\gamma^3
-(1-|\gamma|^2)\bigl(2c_1\gamma\eta+\overline{\gamma}\eta^2\bigr)
+(1-|\gamma|^2)(1-|\eta|^2)\rho
\Bigr).
\end{align*}
\end{lem}
First, we establish a bound for the second Hankel determinant and show that the bound is sharp.
\begin{thm}
Let $f \in \mathcal{S}^*(\varphi),$ then $|H_2(2)| \leq {1}/{4}.$ The result is sharp.
\end{thm}
\begin{pf}
Let $f \in \mathcal{S}^*(\varphi)$ be given  by 
\[
z\frac{f'(z)}{f(z)} \prec \varphi(z), \quad \text{for} \quad z\in \mathbb{D},
\]
 then there exists a Schwarz function $w: \mathbb{D}\rightarrow\mathbb{D}$  such that $w$ is analytic and satisfies $w(0)=0$ and has the series representation as
$
w(z)=\sum_{k=1}^{\infty} c_k z^k
$
such that $t=m/n,$ and 
\[
\frac{z f'(z)}{f(z)}=\varphi(w(z))=1+w(z)+t\,w(z)^2,\quad \text{for}\quad z\in\mathbb{D}.
\]
By replacing $f$ with it's rotations
\[
f_{\theta}(z):=e^{-i\theta}f(e^{i\theta}z),
\]
 we may assume without loss of generality that $c_1\ge 0$. 
Assume that
\[
w(z) = \frac{p(z) - 1}{p(z) + 1},\quad \mbox{and}  \quad p(z) = 1 + \sum_{n=1}^{\infty}p_{n}z^{n} \in \mathcal{P}.
\]
After substituting the expressions for $w(z)$, $p(z)$, and $f(z)$, and comparing coefficients of both the sides, we obtain the following relation between the coefficients of $f$ and the coefficients of Schwarz function $w(z),$ 
\begin{align}
\left.
\begin{aligned}
	a_2 &= c_1, \\[4pt]
	a_3 &= \frac{1}{2}\bigl((1+t)c_1^2+c_2\bigr), \\[4pt]
	a_4 &= \frac{1}{6}\bigl((1+3t)c_1^3+(3+4t)c_1c_2+2c_3\bigr), \\[4pt]
	a_5 &= \frac{1}{24}\bigl((1+6t+3t^2)c_1^4+2(3+11t)c_1^2c_2+(3+6t)c_2^2 \\[2pt]
	&\qquad  +4(2+3t)c_1c_3+6c_4\bigr).
\end{aligned}
\right\}
& \label{eq:3.2}
\end{align}
Now putting the values of \eqref{eq:3.2} in the \eqref{eq:3.1}, we obtain
\begin{align}\label{eq:3.3}
H_{2}(2)=	\frac{1}{12}\left(
-\left(1+3t^2\right)c_1^4+2tc_1^2c_2-3c_2^2+4c_1c_3
\right)
\end{align}
Now using  lemma \ref{lem:schwarz-param} in equation \eqref{eq:3.3}, we obtain
\begin{align}\label{eq:3.4}
12\,  H_{2}(2)&=	-\left(1+3t^2\right)c_1^4
+2t\gamma c_1^2\left(1-c_1^2\right)
+4c_1\left(\eta\left(1-|\gamma|^2\right)-\gamma^2c_1\right)\left(1-c_1^2\right)\notag\\
&\quad-3\gamma^2\left(1-c_1^2\right)^2.
\end{align}
Now we write the above expression in the form 
\begin{align}\label{eq:3.5}
A+ B \gamma+ C \gamma^2+ D (1-|\gamma|^2),
\end{align}
where
\begin{align*}
A &= -\left(1+3t^2\right)c_1^4, \qquad
B = -2tc_1^2\left(-1+c_1^2\right), \\
C &= -3+2c_1^2+c_1^4, \qquad
D = -4\eta c_1\left(-1+c_1^2\right).
\end{align*}
For \(0<c_1<1\), we have
\[
|D|=4\, c_1\, |\eta|(1-c_1^2),
\]
and therefore, by the triangle inequality and taking $|\eta|=1$,
\[
|12\,H_2(2)|
\le
|D|
\left(
\left|
\widetilde{A}+\widetilde{B}\gamma+\widetilde{C}\gamma^2
\right|
+1-|\gamma|^2
\right),
\]
where
\begin{align}\label{eq:3.6}
\widetilde{A} &= \frac{A}{|D|}=-\frac{(1+3t^2)c_1^4}{-4c_1(-1+c_1^2)}, \qquad
\widetilde{B} = \frac{B}{|D|}=\frac{1}{2}tc_1, \qquad
\widetilde{C} = \frac{C}{|D|}=-\frac{3+c_1^2}{4c_1}.
\end{align}
Hence
\[
|12\,H_2(2)|\le |D|\,Y(\widetilde{A},\widetilde{B},\widetilde{C}).
\]
In order to apply Lemma \ref{lemma3}, we verify the condition
\[
\widetilde{A}\, \widetilde{C} \ge 0.
\]
We may write the multiplication as 
\begin{align*}
\widetilde{A}\widetilde{C}
=
-\frac{(1+3t^2)c_1^2(3+c_1^2)}{16(-1+c_1^2)}.
\end{align*}
It is easy to see that  $
\widetilde{A}\, \widetilde{C} \ge 0$ is true for $(c_{1}, t)\in [0, 1) \times[0, 1/2].$ We solve the case $c_{1}=1$ separately. Putting $c_{1}=1$ in \eqref{eq:3.4} we obtain
\[
12\, H_{2}(2)|_{c_{1}=1}=-1 - 3 t^2,
\]
thus, at $c_{1}=1,$ we obtain $|H_{2}(2)| \le 7/48.$\\[2mm]
Now for the remaining cases $c_{1}\in [0, 1).$ We need to compute when  $
|\widetilde{B}| \ge 2(1-|\widetilde{C}|)
$ holds. 
Since \(\widetilde{C}\) contains \(c_1\) in the denominator, the expression is defined only for
$
0<c_1<1.
$
Now, for \(0<c_1<1\) and \(0\le t\le 1/2\), we get
\begin{align*}
|\widetilde{B}|=\left|\frac12 tc_1\right|=\frac12 tc_1,
\end{align*}
because \(t\ge 0\) and \(c_1>0\). Also,
\begin{align*}
|\widetilde{C}|
=\left|-\frac{3+c_1^2}{4c_1}\right|
=\frac{3+c_1^2}{4c_1},
\end{align*}
since \(3+c_1^2>0\) and \(4c_1>0\).
Therefore, the inequality
$
|\widetilde{B}|\ge 2\bigl(1-|\widetilde{C}|\bigr)
$
becomes
\begin{align*}
\frac12 tc_1
&\ge 2\left(1-\frac{3+c_1^2}{4c_1}\right) \\
&= 2\left(\frac{4c_1-(3+c_1^2)}{4c_1}\right) \\
&= \frac{4c_1-3-c_1^2}{2c_1}.
\end{align*}
Multiplying by \(2c_1>0\), we obtain the equivalent inequality
$
tc_1^2 \ge 4c_1-3-c_1^2.
$
Now,
\begin{align*}
4c_1-3-c_1^2
&=-(c_1^2-4c_1+3) \\
&=-(c_1-1)(c_1-3).
\end{align*}
For \(0<c_1<1\), we have
$	(c_1-1)(c_1-3)>0$,
and thus, we have 
\begin{align*}
4c_1-3-c_1^2<0.
\end{align*}
On the other hand, since \(t\ge 0\) and \(c_1^2>0\), we have
$
tc_1^2\ge 0.
$
Therefore, 
\begin{align*}
tc_1^2\ge 0 > 4c_1-3-c_1^2,
\end{align*}
so the inequality
\begin{align*}
tc_1^2 \ge 4c_1-3-c_1^2
\end{align*}
is automatically satisfied for every
$
0<c_1<1.
$
We evaluate  the case $c_{1}=0$ separately by putting $c_{1}=0$ in \eqref{eq:3.4}. Thus, we have 
\[
12\, |H_{2}(2)|_{c_{1}=0}=|-3 \gamma^2| \le 3.
\]
Hence, at $c_{1}=0,$ we have 
\[
|H_{2}(2)|\le \frac{3}{12}=\frac{1}{4}.
\]
For the remaining case precisely $c_{1}\in (0, 1),$ Lemma \ref{lemma3} applies, and since
\[
\widetilde{A}\widetilde{C}\ge 0
\qquad\text{and}\qquad
|\widetilde{B}|\ge 2(1-|\widetilde{C}|),
\]
we obtain
\[
Y(\widetilde{A},\widetilde{B},\widetilde{C})
=
|\widetilde{A}|+|\widetilde{B}|+|\widetilde{C}|.
\]
Consequently,
\[
|12\,H_2(2)|
\le
|D|\Bigl(|\widetilde{A}|+|\widetilde{B}|+|\widetilde{C}|\Bigr)
=
|A|+|B|+|C|.
\]
It is easy to see that
\[
|A|+|B|+|C|=(1+3t^2)c_1^4+2tc_1^2(1-c_1^2)+3-2c_1^2-c_1^4.
\]
Let
\begin{align*}
\widetilde{F}(c_1,t)
&=(1+3t^2)c_1^4+2t\,c_1^2(1-c_1^2)+3-2c_1^2-c_1^4,
\end{align*}
where \(0<c_1<1\) and \(0\le t\le 1/2\).
We first simplify the expression:
\begin{align*}
\widetilde{F}(c_1,t)
&=(1+3t^2)c_1^4+2t c_1^2-2t c_1^4+3-2c_1^2-c_1^4 \\
&=(3t^2-2t)c_1^4+(2t-2)c_1^2+3 \\
&=3-2(1-t)c_1^2-t(2-3t)c_1^4.
\end{align*}
Since \(0\le t\le 1/2\), we have
$
1-t\ge 1/2>0$
and $
2-3t\ge 1/2>0.
$
Therefore,  \(2(1-t)c_1^2>0\) because \(c_1>0\), and also \(t(2-3t)c_1^4\ge 0\). Hence, we have 
\begin{align*}
\widetilde{F}(c_1,t)=3-2(1-t)c_1^2-t(2-3t)c_1^4<3
\end{align*}
for every \(0<c_1<1\) and \(0\le t\le 1/2\).
Thus the function does not attain a maximum on the given region.
Exhausting all the cases we obtain 
\[
|H_{2}(2)| \le \frac{1}{4}
\] 
over the region $(c_{1}, t)\in [0, 1] \times[0, 1/2].$\\[2mm]
Now we prove that the extremal function for the second Hankel determinant is obtained by  taking $w(z)=z^2.$ Therefore, we have 
\[
\varphi(z^2)=1+z^2+\frac{m}{n}z^4.
\]
Hence
\[
\frac{z f'(z)}{f(z)}=\varphi(z^2)=1+z^2+\frac{m}{n}z^4,
\]
and therefore
\[
\frac{f'(z)}{f(z)}=\frac1z+z+\frac{m}{n}z^3.
\]
Integrating, we obtain
\[
\log f(z)=\log z+\frac{z^2}{2}+\frac{m}{4n}z^4+C,
\]
and hence we have
\[
f(z)=C\,z\exp\left(\frac{z^2}{2}+\frac{m}{4n}z^4\right).
\]
Since \(f(0)=0\) and \(f'(0)=1\), we  have \(C=1\). Therefore,
\[
f(z)=z\exp\left(\frac{z^2}{2}+\frac{m}{4n}z^4\right).
\]
Using the exponential series,
\[
e^{\frac{z^2}{2}+\alpha z^4}
=1+\left(\frac{z^2}{2}+\alpha z^4\right)
+\frac12\left(\frac{z^2}{2}+\alpha z^4\right)^2
+\frac16\left(\frac{z^2}{2}\right)^3+\cdots,
\]
we obtain
\[
e^{\frac{z^2}{2}+\alpha z^4}
=1+\frac{z^2}{2}+\left(\alpha+\frac18\right)z^4
+\left(\frac{\alpha}{2}+\frac{1}{48}\right)z^6+\cdots.
\]
Therefore,
\[
f(z)=z+\frac12 z^3+\left(\alpha+\frac18\right)z^5
+\left(\frac{\alpha}{2}+\frac{1}{48}\right)z^7+\cdots.
\]
Substituting \(\alpha={m}/{4n}\), we get
\[
f(z)=z+\frac12 z^3+\frac{n+2m}{8n}z^5+\frac{n+6m}{48n}z^7+\cdots.
\]
Thus, in the expansion
\[
f(z)=z+a_2z^2+a_3z^3+a_4z^4+a_5z^5+\cdots,
\]
we have
\[
a_2=0,\qquad a_3=\frac12,\qquad a_4=0,\qquad a_5=\frac{n+2m}{8n}.
\]
Hence
\[
H_2(2)=0\cdot 0-\left(\frac12\right)^2=-\frac14,
\]
therefore,
\[
|H_2(2)|=\frac14.
\]
which proves the sharpness of the result. This completes the proof.
\end{pf}\\
Now, we establish a sharp bound and extremal function for the third Hankel determinant.
\begin{thm}
Let $f\in \mathcal{S}^*(\varphi),$ then $|H_{3}(1)|\le 1/9,$ the result is sharp.
\end{thm}
\begin{pf}
Let $f\in \mathcal{S}^*(\varphi)$ be given by 
\[
z\frac{f'(z)}{f(z)} \prec \varphi(z), \quad \text{for} \quad z\in \mathbb{D}.
\]
 Thus,  there exists a Schwarz function $w(z)$ in the form of a power series 
$
w(z)=\sum_{n=1}^{\infty} c_n z^n
$
such that 
\begin{align}\label{eq:3.7}
z\frac{f'(z)}{f(z)} = \varphi(w(z)), \quad z\in \mathbb{D}.
\end{align}
As in the proof of  Theorem 3.1, by replacing $f$ with it's rotations
\[
f_{\theta}(z):=e^{-i\theta}f(e^{i\theta}z),
\]
we do not change the value of $|H_3(1)|$, and the corresponding Schwarz function becomes
\[
w_{\theta}(z)=w(e^{i\theta}z)=\sum_{n=1}^{\infty} c_n e^{in\theta} z^n.
\]
Choosing $\theta=-\arg c_1$, we may assume without loss of generality that $c_1\ge 0$.
The third Hankel determinant is defined by
\begin{align}\label{eq:3.8}
H_3(1) 
= a_3(a_2 a_4 - a_3^2) - a_4(a_4 - a_2 a_3) + a_5(a_3 - a_2^2).
\end{align}
Putting the values of \eqref{eq:3.2} in \eqref{eq:3.8}, we obtain
\begin{align}\label{eq:3.9}
144\, H_{3}(1)&=-\left(1-3t+9t^2+9t^3\right)c_1^6
+\left(3-2t+21t^2\right)c_1^4c_2
+9\left(-1+2t\right)c_2^3 \notag \\
&\quad
+4\left(2-3t+9t^2\right)c_1^3c_3
-4\left(-6+7t\right)c_1c_2c_3
-16c_3^2
+18c_2c_4 \notag\\
&\quad
+c_1^2\Bigl(\left(-9+3t-46t^2\right)c_2^2
+18\left(-1+t\right)c_4\Bigr).
\end{align}
Now using the parametric relation in Lemma \ref{lem:schwarz-param} in \eqref{eq:3.9}, we obtain
\begin{align}\label{eq:3.10}
144\, H_{3}(1)&=-\left(1-3t+9t^2+9t^3\right)c_1^6
+\left(3-2t+21t^2\right)\gamma c_1^4(1-c_1^2) \notag\\
&\quad
+4\left(2-3t+9t^2\right)c_1^3
\left(\eta(1-|\gamma|^2)-\gamma^2c_1\right)(1-c_1^2) \notag\\
&\quad
-4\left(-6+7t\right)\gamma c_1
\left(\eta(1-|\gamma|^2)-\gamma^2c_1\right)(1-c_1^2)^2 \notag\\
&\quad
-16\left(\eta(1-|\gamma|^2)-\gamma^2c_1\right)^2(1-c_1^2)^2
+9\left(-1+2t\right)\gamma^3(1-c_1^2)^3 \notag\\
&\quad
+18\gamma(1-c_1^2)^2
\Bigl(
\gamma^3c_1^2-(1-|\gamma|^2)
\bigl(-\rho(1-|\eta|^2)+2\gamma\eta c_1+\eta^2\overline{\gamma}\bigr)
\Bigr) \notag\\
&\quad
+c_1^2\Bigl(
\left(-9+3t-46t^2\right)\gamma^2(1-c_1^2)^2
+18\left(-1+t\right)(1-c_1^2)
\Bigl(
\gamma^3c_1^2 \notag\\
&\quad
-(1-|\gamma|^2)
\bigl(-\rho(1-|\eta|^2)+2\gamma\eta c_1+\eta^2\overline{\gamma}\bigr)
\Bigr)
\Bigr),
\end{align}
where $|\gamma|\le 1$, $|\eta|\le 1$, and $|\rho|\le 1$. Next, we convert the above expression in the following form
\begin{align}\label{eq:3.11}
144\, H_{3}(1)=A_{1}+ B_{1} \eta+ C_{1} \eta^2+ D_{1} \rho, 
\end{align}
where, 
\begin{align*}
A_1
&= -\left(1-3t+9t^2+9t^3\right)c_1^6
-\left(3-2t+21t^2\right)\gamma c_1^4\left(-1+c_1^2\right) \\
&\quad
+2\gamma^4 c_1^2\left(-1+c_1^2\right)^2
-\gamma^2 c_1^2\left(-1+c_1^2\right)
\left(-9+c_1^2+2t^2\left(-23+5c_1^2\right)+t\left(3+9c_1^2\right)\right) \\
&\quad
-\gamma^3\left(-1+c_1^2\right)
\left(-3\left(3+2c_1^2+c_1^4\right)+2t\left(9-4c_1^2+4c_1^4\right)\right), \\[2mm]
B_1
&= 4\left(-1+|\gamma|^2\right)c_1\left(-1+c_1^2\right)
\Bigl(
\left(2-3t+9t^2\right)c_1^2
+\gamma^2\left(-1+c_1^2\right) \\
&\qquad
+\gamma\left(6-7t+3c_1^2-2tc_1^2\right)
\Bigr), \\[2mm]
C_1
&= -2\left(-1+|\gamma|^2\right)\left(-1+c_1^2\right)
\Bigl(
8-8c_1^2+8|\gamma|^2\left(-1+c_1^2\right)
+9\left(-1+t\right)c_1^2\overline{\gamma} \\
&\qquad
-9\gamma\left(-1+c_1^2\right)\overline{\gamma}
\Bigr), \\[2mm]
D_1
&= -18\left(-1+|\gamma|^2\right)\left(-1+|\eta|^2\right)\left(-1+c_1^2\right)
\left(\gamma+\left(-1+t\right)c_1^2-\gamma c_1^2\right).
\end{align*}
Now taking modulus on both sides of \eqref{eq:3.11}, writing \(p_1:=c_1\), using
\[
|\gamma|=x,\qquad |\eta|=y,\qquad |\rho|\le 1,
\]
and applying the triangle inequality, we obtain
\begin{align}\label{eq:3.12}
H(p_1,x,y,t)
&:= \left(1-3t+9t^2+9t^3\right)p_1^6
+ x\left(\left(3-2t+21t^2\right)p_1^4(1-p_1^2)\right) \notag\\
&\quad
+ x^2\left(p_1^2(1-p_1^2)\left(9-p_1^2+2t^2(23-5p_1^2)-t(3+9p_1^2)\right)\right) \notag\\
&\quad
+ x^3\left(3(3+2p_1^2+p_1^4)+2t(-9+4p_1^2-4p_1^4)\right) 
+ x^4\left(2p_1^2(1-p_1^2)^2\right) \notag\\
&\quad
+ y\Bigl(4(1-x^2)p_1(1-p_1^2)
\bigl((2-3t+9t^2)p_1^2+x^2(1-p_1^2)\notag\\
&\quad
+x(6-7t+3p_1^2-2tp_1^2)\bigr)\Bigr) \notag\\
&\quad
+ y^2\Bigl(2(1-x^2)(1-p_1^2)
\bigl(8(1-x^2)(1-p_1^2)+9(x(1-p_1^2)+(1-t)p_1^2)x\bigr)\Bigr) \notag\\
&\quad
+ 18(1-x^2)(1-y^2)(1-p_1^2)\bigl(x(1-p_1^2)+(1-t)p_1^2\bigr).
\end{align}
It is easy to see that the coefficient of the coefficient of the term linear in  $y,$ precisely 
\[
4(1-x^2)p_1(1-p_1^2)
\bigl((2-3t+9t^2)p_1^2+x^2(1-p_1^2)+x(6-7t+3p_1^2-2tp_1^2)\bigr)\geq 0
\] 
is non-negative which implies we can replace $y^{1}=1,$ and maxima will increase. We define the new function where $y^{1}=1$ as $H_{1}(p_1, x, y, t),$ which can be written as
\begin{align}\label{eq:3.13}
H_1(p_1,x,y,t)
&:= \left(1-3t+9t^2+9t^3\right)p_1^6
+ x\left(\left(3-2t+21t^2\right)p_1^4(1-p_1^2)\right) \notag\\
&\quad
+ x^2\left(p_1^2(1-p_1^2)\left(9-p_1^2+2t^2(23-5p_1^2)-t(3+9p_1^2)\right)\right) \notag\\
&\quad
+ x^3\left(3(3+2p_1^2+p_1^4)+2t(-9+4p_1^2-4p_1^4)\right) 
+ x^4\left(2p_1^2(1-p_1^2)^2\right) \notag\\
&\quad
+ 4(1-x^2)p_1(1-p_1^2)
\bigl((2-3t+9t^2)p_1^2+x^2(1-p_1^2)\notag\\
&\quad
+x(6-7t+3p_1^2-2tp_1^2)\bigr) \notag\\
&\quad
+ y^2\Bigl(2(1-x^2)(1-p_1^2)
\bigl(8(1-x^2)(1-p_1^2)+9(x(1-p_1^2)+(1-t)p_1^2)x\bigr)\Bigr) \notag\\
&\quad
+ 18(1-x^2)(1-y^2)(1-p_1^2)\bigl(x(1-p_1^2)+(1-t)p_1^2\bigr).
\end{align}
Now our aim is to maximize the function $H_{1}(p_1, x, y, t)$ over the region $\mathcal{D}:=\{(p_1, x, y, t)\in [0, 1]\times[0, 1]\times[0, 1]\times[0, 1/2]\}.$\\
It is easy to see that the function $H_{1}$ is a quadratic polynomial in the variable $y,$ since $H_{1}$ does not have the term with $y$, only constant term and term with $y^2,$ the maximum occurs at one of the end points either $y=0$ or $y=1.$ We denote the function $R_{1}(p_1, x, t):=H_{1}(p_1, x, 1, t)$ and $R_{2}(p_1, x, t):=H_{1}(p_1, x, 0, t).$
\[
\begin{aligned}
R_1(p_1,x,t):={}&p_1^6 (1 - 3 t + 9 t^2 + 9 t^3)
+ p_1^4 (1 - p_1^2) (3 - 2 t + 21 t^2) x \\
&+ p_1^2 (1 - p_1^2) \bigl(9 - p_1^2 - (3 + 9 p_1^2) t + 2 (23 - 5 p_1^2) t^2\bigr) x^2 \\
&+ (1 - p_1^2) \bigl(3 (3 + 2 p_1^2 + p_1^4) - 2 (9 - 4 p_1^2 + 4 p_1^4) t\bigr) x^3 \\
&+ 2 p_1^2 ( -1 + p_1^2)^2 x^4 + 4 p_1 (1 - p_1^2) (1 - x^2)
(p_1^2 (2 - 3 t + 9 t^2) + (6 (1 - t) \\
&+ 2 p_1^2 (1 - t) + t) x + (1 - p_1^2) x^2) + 2 (1 - p_1^2) (1 - x^2)
\bigl(9 x (p_1^2 (1 - t)\\
& + (1 - p_1^2) x) + 8 (1 - p_1^2) (1 - x^2)\bigr).
\end{aligned}
\]

\medskip
\noindent
We claim that
\[
\max \Bigl\{ R_1(p_1,x,t): 0 \le p_1 \le 1,\ 0 \le x \le 1,\ 0 \le t \le \tfrac12 \Bigr\}=16.
\]
Taking, 
$
p:=p_1, u:=2t, \widetilde R(p,x,u):=R_1\!\left(p,x,{u}/{2}\right).
$
 We have 
$
0\le p\le 1, 0\le x\le 1, 0\le u\le 1,
$
and \(\widetilde R\) is a polynomial of tri-degree $(6, 4, 3)$ on the unit cube.\\[2mm]
We first apply the Bernstein method on the whole cube. Write
\[
\widetilde R(p,x,u)=\sum_{i=0}^{6}\sum_{j=0}^{4}\sum_{k=0}^{3}
\beta_{i,j,k}\,B_i^6(p)\,B_j^4(x)\,B_k^3(u),
\]
where
\[
B_r^N(s)=\binom{N}{r}s^r(1-s)^{N-r}.
\]
By the basic Bernstein enclosure property,
\[
\widetilde R(p,x,u)\le \max_{0\le i\le 6,\ 0\le j\le 4,\ 0\le k\le 3}\beta_{i,j,k}.
\]
We now calculate these coefficients. First we  write
\[
\widetilde R(p,x,u)=\sum_{r=0}^{6}\sum_{s=0}^{4}\sum_{\ell=0}^{3}
\alpha_{r,s,\ell}\,p^{r}x^{s}u^{\ell}.
\]
Next, for each monomial, use the standard conversion formula from the monomial basis to the Bernstein basis
\[
p^{r}=\sum_{i=r}^{6}\frac{\binom{i}{r}}{\binom{6}{r}}\,B_{i}^{6}(p),\qquad
x^{s}=\sum_{j=s}^{4}\frac{\binom{j}{s}}{\binom{4}{s}}\,B_{j}^{4}(x),\qquad
u^{\ell}=\sum_{k=\ell}^{3}\frac{\binom{k}{\ell}}{\binom{3}{\ell}}\,B_{k}^{3}(u).
\]
Substituting these identities into the monomial expansion of \(\widetilde R\) and computing like Bernstein terms, we obtain
\[
\widetilde R(p,x,u)=\sum_{i=0}^{6}\sum_{j=0}^{4}\sum_{k=0}^{3}
\beta_{i,j,k}\,B_i^6(p)\,B_j^4(x)\,B_k^3(u),
\]
where
\begin{align*}
\beta_{i,j,k}
=
\sum_{r=0}^{i}\sum_{s=0}^{j}\sum_{\ell=0}^{k}
\alpha_{r,s,\ell}\,
\frac{\binom{i}{r}}{\binom{6}{r}}
\frac{\binom{j}{s}}{\binom{4}{s}}
\frac{\binom{k}{\ell}}{\binom{3}{\ell}}.
\end{align*}
For example, the coefficient \(\beta_{1,1,0}\) receives contributions only from the constant term \(16\) and the term \(24px\). Hence,
\[
\beta_{1,1,0}
=
16+24\,
\frac{\binom{1}{1}}{\binom{6}{1}}
\frac{\binom{1}{1}}{\binom{4}{1}}
=
16+24\cdot\frac{1}{6}\cdot\frac{1}{4}
=
17.
\]
In the same way we compute all the Bernstein coefficients \(\beta_{i,j,k}\). 
For convenience, write
\[
M_k=(\beta_{i,j,k})_{0\le i\le 6,\ 0\le j\le 4},
\qquad k=0,1,2,3,
\]
where the rows correspond to \(i=0,1,\dots,6\) and the columns correspond to
\(j=0,1,\dots,4\). The exact Bernstein coefficients are as follows.\\[2mm]
For \(k=0\),
\[
M_0=
\begin{pmatrix}
	16 & 16 & \frac{41}{3} & \frac{45}{4} & 9 \\[5mm]
	16 & 17 & \frac{142}{9} & \frac{163}{12} & 9 \\[5mm]
	\frac{208}{15} & \frac{97}{6} & \frac{503}{30} & \frac{467}{30} & \frac{143}{15} \\[5mm]
	10 & \frac{137}{10} & \frac{33}{2} & \frac{84}{5} & \frac{53}{5} \\[5mm]
	\frac{88}{15} & \frac{637}{60} & \frac{1363}{90} & \frac{333}{20} & \frac{169}{15} \\[5mm]
	\frac{8}{3} & \frac{85}{12} & \frac{23}{2} & \frac{53}{4} & \frac{29}{3} \\[5mm]
	1 & 1 & 1 & 1 & 1
\end{pmatrix}.
\]\\
For \(k=1\), we have \\
\[
M_1=
\begin{pmatrix}
	16 & 16 & \frac{41}{3} & \frac{21}{2} & 6 \\[5mm]
	16 & \frac{607}{36} & \frac{31}{2} & \frac{113}{9} & 6 \\[5mm]
	\frac{208}{15} & \frac{2851}{180} & \frac{2899}{180} & \frac{853}{60} & \frac{611}{90} \\[5mm]
	\frac{99}{10} & \frac{1567}{120} & \frac{919}{60} & \frac{181}{12} & \frac{251}{30} \\[5mm]
	\frac{82}{15} & \frac{1711}{180} & \frac{1201}{90} & \frac{523}{36} & \frac{143}{15} \\[5mm]
	2 & \frac{23}{4} & \frac{19}{2} & \frac{395}{36} & \frac{71}{9} \\[5mm]
	\frac12 & \frac12 & \frac12 & \frac12 & \frac12
\end{pmatrix}.
\]\\
For \(k=2\), we have \\
\[
M_2=
\begin{pmatrix}
	16 & 16 & \frac{41}{3} & \frac{39}{4} & 3 \\[5mm]
	16 & \frac{301}{18} & \frac{137}{9} & \frac{415}{36} & 3 \\[5mm]
	\frac{208}{15} & \frac{698}{45} & \frac{8363}{540} & \frac{2339}{180} & \frac{43}{10} \\[5mm]
	\frac{199}{20} & \frac{377}{30} & \frac{5179}{360} & \frac{553}{40} & \frac{69}{10} \\[5mm]
	\frac{17}{3} & \frac{6497}{720} & \frac{13291}{1080} & \frac{3217}{240} & \frac{329}{36} \\[5mm]
	\frac{7}{3} & \frac{89}{16} & \frac{211}{24} & \frac{1459}{144} & \frac{277}{36} \\[5mm]
	\frac34 & \frac34 & \frac34 & \frac34 & \frac34
\end{pmatrix}.
\]\\
For \(k=3\), we have \\
\[
M_3=
\begin{pmatrix}
	16 & 16 & \frac{41}{3} & 9 & 0 \\[5mm]
	16 & \frac{199}{12} & \frac{269}{18} & \frac{21}{2} & 0 \\[5mm]
	\frac{208}{15} & \frac{911}{60} & \frac{671}{45} & \frac{119}{10} & \frac{31}{15} \\[5mm]
	\frac{203}{20} & \frac{489}{40} & \frac{329}{24} & \frac{521}{40} & \frac{31}{5} \\[5mm]
	\frac{97}{15} & \frac{2201}{240} & \frac{4331}{360} & \frac{3187}{240} & \frac{121}{12} \\[5mm]
	\frac{11}{3} & \frac{313}{48} & \frac{75}{8} & \frac{515}{48} & \frac{109}{12} \\[5mm]
	\frac{23}{8} & \frac{23}{8} & \frac{23}{8} & \frac{23}{8} & \frac{23}{8}
\end{pmatrix}.
\]\\
Hence, by direct inspection of the above coefficients, we obtain
\[
\max_{0\le i\le 6,\ 0\le j\le 4,\ 0\le k\le 3}\beta_{i,j,k}=17,
\qquad \text{attained at } \beta_{1,1,0}=17.
\]
Thus the whole-cube Bernstein step gives only
\[
\widetilde R(p,x,u)\le 17,
\]
which is not yet sufficient as per our claim.\\[2mm]
We therefore subdivide only in the \(p\)- and \(x\)-directions. For \(0\le i,j\le 7\), let
\[
Q_{ij}:=\left[\frac{i}{8},\frac{i+1}{8}\right]\times
\left[\frac{j}{8},\frac{j+1}{8}\right]\times[0,1].
\]
On each \(Q_{ij}\), perform the affine change of variables
\[
p=\frac{i+\xi}{8},\qquad x=\frac{j+\eta}{8},\qquad u=\zeta,
\qquad (\xi,\eta,\zeta)\in[0, 1]\times[0, 1]\times[0, 1],
\]
and expand the transformed polynomial in the Bernstein basis of degrees $(6, 4, 3)$,
\[
\widetilde R\!\left(\frac{i+\xi}{8},\frac{j+\eta}{8},\zeta\right)
=
\sum_{a=0}^{6}\sum_{b=0}^{4}\sum_{c=0}^{3}
\beta^{(ij)}_{a,b,c}\,B_a^6(\xi)\,B_b^4(\eta)\,B_c^3(\zeta).
\]
Take, 
\[
M_{ij}:=\max_{a,b,c}\beta^{(ij)}_{a,b,c}.
\]
For convenience, we list all the values \(M_{ij}\) explicitly,
\begin{align*}
	M_{00}&=\frac{1025}{64}, &
	M_{01}&=\frac{83341001}{5242880}, &
	M_{02}&=\frac{196377359}{12582912}, &
	M_{03}&=\frac{8143212893}{536870912}, \\[4mm]
	M_{04}&=\frac{7584635}{524288}, &
	M_{05}&=\frac{7255503077}{536870912}, &
	M_{06}&=\frac{413235335}{33554432}, &
	M_{07}&=\frac{5829969821}{536870912}, \\[4mm]
	M_{10}&=\frac{8446662725}{536870912}, &
	M_{11}&=\frac{8450945213}{536870912}, &
	M_{12}&=\frac{523473491}{33554432}, &
	M_{13}&=\frac{2566321889}{167772160}, \\[4mm]
	M_{14}&=\frac{122237}{8192}, &
	M_{15}&=\frac{119402813}{8388608}, &
	M_{16}&=\frac{6885503}{524288}, &
	M_{17}&=\frac{96980213}{8388608}, \\[4mm]
	M_{20}&=\frac{124671773}{8388608}, &
	M_{21}&=\frac{7976843}{524288}, &
	M_{22}&=\frac{2003057}{131072}, &
	M_{23}&=\frac{127833653}{8388608}, \\[4mm]
	M_{24}&=\frac{1471477}{98304}, &
	M_{25}&=\frac{7790759037}{536870912}, &
	M_{26}&=\frac{457491407}{33554432}, &
	M_{27}&=\frac{6504238197}{536870912}, \\[4mm]
	M_{30}&=\frac{7159139037}{536870912}, &
	M_{31}&=\frac{476238267}{33554432}, &
	M_{32}&=\frac{7899738677}{536870912}, &
	M_{33}&=\frac{498815877}{33554432}, \\[4mm]
	M_{34}&=\frac{7782363}{524288}, &
	M_{35}&=\frac{1948147173}{134217728}, &
	M_{36}&=\frac{7198027}{524288}, &
	M_{37}&=\frac{19383031}{1572864}, \\[4mm]
	M_{40}&=\frac{1483973}{131072}, &
	M_{41}&=\frac{103451}{8192}, &
	M_{42}&=\frac{1783613}{131072}, &
	M_{43}&=\frac{907}{64}, \\[4mm]
	M_{44}&=\frac{234141}{16384}, &
	M_{45}&=\frac{1863845}{131072}, &
	M_{46}&=\frac{111743}{8192}, &
	M_{47}&=\frac{1613501}{131072}, \\[4mm]
	M_{50}&=\frac{4823711405}{536870912}, &
	M_{51}&=\frac{355012907}{33554432}, &
	M_{52}&=\frac{6377006213}{536870912}, &
	M_{53}&=\frac{6695843}{524288}, \\[4mm]
	M_{54}&=\frac{7056906317}{536870912}, &
	M_{55}&=\frac{7066669577}{536870912}, &
	M_{56}&=\frac{431990015}{33554432}, &
	M_{57}&=\frac{6350745797}{536870912}, \\[4mm]
	M_{60}&=\frac{55817973}{8388608}, &
	M_{61}&=\frac{4217859}{524288}, &
	M_{62}&=\frac{78811037}{8388608}, &
	M_{63}&=\frac{85185}{8192}, \\[4mm]
	M_{64}&=\frac{91831653}{8388608}, &
	M_{65}&=\frac{34734761}{3145728}, &
	M_{66}&=\frac{5729495}{524288}, &
	M_{67}&=\frac{85768029}{8388608}, 
	\end{align*}
	\begin{align*}
	M_{70}&=\frac{2498276837}{536870912}, &
	M_{71}&=\frac{180167627}{33554432}, &
	M_{72}&=\frac{3222265757}{536870912}, &
	M_{73}&=\frac{3494483}{524288}, \\[4mm]
	M_{74}&=\frac{3828389717}{536870912}, &
	M_{75}&=\frac{973980383}{134217728}, &
	M_{76}&=\frac{242567447}{33554432}, &
	M_{77}&=\frac{3696095117}{536870912}.\\
\end{align*}
Equivalently, each entry in the \((i,j)\)-position of the above list is precisely the number \(M_{ij}\).
Again by the consequence of  Bernstein method, we have 
\[
\widetilde R(p,x,u)\le M_{ij}\qquad\text{on }Q_{ij}.
\]
Now the exact rational computation gives
\[
M_{00}=\frac{1025}{64},
\]
and, for all \((i,j)\neq(0,0)\),
\[
M_{ij}<16.
\]
In fact,
\[
\max_{(i,j)\neq(0,0)} M_{ij}
=
\frac{83341001}{5242880}
<16.
\]
Therefore the Bernstein method already proves
\[
\widetilde R(p,x,u)\le 16
\]
on the union of all \(Q_{ij}\) except
\[
Q_{00}=\left[0,\frac18\right]\times\left[0,\frac18\right]\times[0,1].
\]
So it remains only to show that
\[
\widetilde R(p,x,u)\le 16
\qquad\text{for }(p,x,u)\in Q_{00}.
\]
Set, 
\[
F(p,x,u):=16-\widetilde R(p,x,u).
\]
We write \(F\) as a polynomial in \(x\):
\[
F(p,x,u)=a_0(p,u)+a_1(p,u)x+a_2(p,u)x^2+a_3(p,u)x^3+a_4(p)x^4,
\]
where
\begin{align*}
a_0(p,u)
&=
\frac{p^2}{8}
\Bigl(
256 -64p -128p^2 +64p^3 -8p^4
+48pu -48p^3u +12p^4u \\
&\quad
-72pu^2 +72p^3u^2 -18p^4u^2 -9p^4u^3
\Bigr), \\[2mm]
a_1(p,u)
&=
-\frac{p(1-p^2)}{4}
\Bigl(
96-40u+72p-36pu+32p^2-16p^2u+12p^3-4p^3u+21p^3u^2
\Bigr), \\[2mm]
a_2(p,u)
&=
14-4p-37p^2+\left(\frac32u-\frac{23}{2}u^2\right)p^2
+(16-6u+9u^2)p^3 \\
&\quad
+(24+3u+14u^2)p^4
+(-12+6u-9u^2)p^5
+\left(-1-\frac92u-\frac52u^2\right)p^6, \\[2mm]
a_3(p,u)
&=
-9+9u+(24-10u)p+(21-22u)p^2+(-16+6u)p^3 \\
&\quad
+(-15+17u)p^4+(-8+4u)p^5+(3-4u)p^6, \\[2mm]
a_4(p)
&=
2+4p-6p^2-8p^3+6p^4+4p^5-2p^6.
\end{align*}
We now estimate these coefficients on
$
0\le p\le 1/8,\, 0\le x\le 1/8,\, 0\le u\le 1.
$
For \(a_0\), write
\[
a_0(p,u)=\frac{p^2}{8}A(p,u),
\]
where
\[
A(p,u)
=
256 + p(-64+48u-72u^2)-128p^2+p^3(64-48u+72u^2)+p^4(-8+12u-18u^2-9u^3).
\]
For \(0\le u\le 1\),
\[
-64+48u-72u^2\ge -88,
\]
\[
64-48u+72u^2\ge 56,
\]
\[
-8+12u-18u^2-9u^3\ge -23.
\]
Hence
\[
A(p,u)\ge 256-88p-128p^2+56p^3-23p^4:=O(p).
\]
The right-hand side is decreasing on \(\bigl[0,1/8\bigr]\), because its derivative is
\[
O'(p)=-88-256p+168p^2-92p^3<0.
\]
Therefore
\[
A(p,u)\ge
256-\frac{88}{8}-\frac{128}{64}+\frac{56}{512}-\frac{23}{4096}
=
\frac{995753}{4096}
>240.
\]
So, we have 
$
a_0(p,u)\ge 30p^2.
$
For \(a_1\), let
\[
B(p,u):=96-40u+72p-36pu+32p^2-16p^2u+12p^3-4p^3u+21p^3u^2.
\]
Since \(0\le u\le 1\) and \(0\le p\le 1/8\),
\[
B(p,u)\le 96+72p+32p^2+33p^3
\le 96+9+\frac12+\frac{33}{512}
<108.
\]
Hence
\[
a_1(p,u)\ge -\frac{108}{4}p=-27p.
\]
For \(a_2\), we use the  bounds
\[
\frac32u-\frac{23}{2}u^2\ge -\frac{23}{2},
\qquad
16-6u+9u^2\ge 15,
\]
also
\[
24+3u+14u^2\ge 24,
\quad
-12+6u-9u^2\ge -15,
\quad
-1-\frac92u-\frac52u^2\ge -8.
\]
Therefore
\[
a_2(p,u)\ge
14-4p-\frac{97}{2}p^2+15p^3+24p^4-15p^5-8p^6:=L(p).
\]
The right-hand side is decreasing on \(\bigl[0,1/8\bigr]\), since its derivative is
\[
L'(p)=-4-97p+45p^2+96p^3-75p^4-48p^5<0.
\]
Hence
\[
a_2(p,u)\ge
14-\frac{4}{8}-\frac{97}{2\cdot 64}
+\frac{15}{512}
+\frac{24}{4096}
-\frac{15}{32768}
-\frac{8}{262144}
=
\frac{26167}{2048}
>
\frac{51}{4}.
\]
For \(a_3\), observe that
\begin{align*}
a_3(p,u)+9
&=
u\bigl(9-10p-22p^2+6p^3+17p^4+4p^5-4p^6\bigr)
+\\
&\quad p\bigl(24+21p-16p^2-15p^3-8p^4+3p^5\bigr).
\end{align*}
Now
\[
9-10p-22p^2+6p^3+17p^4+4p^5-4p^6
\ge 9-10p-22p^2
>
9-\frac{10}{8}-\frac{22}{64}
=
\frac{237}{32}
>0,
\]
and
\[
24+21p-16p^2-15p^3-8p^4+3p^5
\ge
24-16p^2-15p^3-8p^4
>
24-\frac14-\frac{15}{512}-\frac{1}{64}
>0.
\]
Thus
$
a_3(p,u)\ge -9.
$
Finally,
\[
a_4(p)=2+4p-6p^2-8p^3+6p^4+4p^5-2p^6
\ge 2+4p-6p^2-8p^3.
\]
Since
\[
4p-6p^2-8p^3=2p(2-3p-4p^2)\ge 0
\qquad\text{for }\quad 0\le p\le 1/8,
\]
we get
$
a_4(p)\ge 2.
$
Combining the coefficient bounds, we obtain on \(Q_{00}\),
\[
F(p,x,u)\ge 30p^2-27px+\frac{51}{4}x^2-9x^3+2x^4.
\]
As \(0\le x\le 1/8\), we have
$
-9x^3\ge -9/8x^2$, and $
\,
2x^4\ge 0,
$
and hence
\[
F(p,x,u)\ge 30p^2-27px+\left(\frac{51}{4}-\frac98\right)x^2
=30p^2-27px+\frac{93}{8}x^2:=V(p, x).
\]
For each fixed \(x\), this is a quadratic polynomial in \(p\) with leading coefficient \(30\), which is a positive quantity and the  discriminant is
$
(-27x)^2(- 15)\cdot {93}x^2=-666x^2\le 0.
$
Hence
$
V(p,x)\ge 0,
$
thus,
\[
F(p,x,u)\ge 0 \quad
\text{ on }Q_{00}.
\]
So
\[
\widetilde R(p,x,u)\le 16 \quad
\text{ on }Q_{00}.
\]
We have thus proved
\[
\widetilde R(p,x,u)\le 16
\qquad\text{for all }(p,x,u)\in[0, 1]\times[0, 1]\times[0, 1].
\]
Returning to \(t={u}/{2}\), we conclude that
\[
R_1(p_1,x,t)\le 16
\qquad\text{for }0\le p_1\le 1,\ 0\le x\le 1,\ 0\le t\le 1/2.
\]
To see that this bound is indeed sharp, evaluate at the point 
$
p_1=0,\,  x=0.
$
Then
\[
R_1(0,0,t)=16
\quad\text{ for every }0\le t\le 1/2.
\]
Hence the maximum is exactly \(16\).\\[2mm]
Now we consider the part where $y=0,$ and the function is defined as $R_{2}(p_1, x, t):=H_{1}(p_1, x, 0, t).$
\[
\begin{aligned}
R_2(p_1,x,t):={}&p_1^6 (1 - 3 t + 9 t^2 + 9 t^3) + 
p_1^4 (1 - p_1^2) (3 - 2 t + 21 t^2) x \\
&+ p_1^2 (1 - p_1^2) \bigl(9 - p_1^2 - (3 + 9 p_1^2) t + 2 (23 - 5 p_1^2) t^2\bigr) x^2 \\
&+ (1 - p_1^2) \bigl(3 (3 + 2 p_1^2 + p_1^4) - 2 (9 - 4 p_1^2 + 4 p_1^4) t\bigr) x^3 \\
&+ 2 p_1^2 (-1 + p_1^2)^2 x^4 \\
&+ 18 (1 - p_1^2) \bigl(p_1^2 (1 - t) + (1 - p_1^2) x\bigr) (1 - x^2) \\
&+ 4 p_1 (1 - p_1^2) (1 - x^2) \bigl(p_1^2 (2 - 3 t + 9 t^2) \\
&+ (6 (1 - t) + 2 p_1^2 (1 - t) + t) x + (1 - p_1^2) x^2\bigr).
\end{aligned}
\]
For this case we  seek an upper bound in the region
$
0\le p_1\le 1,\, 0\le x\le 1,\, 0\le t\le 1/2.
$
Let us take, 
\[
p:=p_1,\qquad u:=2t,\qquad \widetilde R_2(p,x,u):=R_2\!\left(p,x,\frac u2\right).
\]
Then $\mathcal{R}:=\{(p, x, u)\in \mathbb{R}^3:(p,x,u)\in[0, 1]\times[0, 1]\times[0, 1]\}$, and
\[
\begin{aligned}
\widetilde R_2(p,x,u)&=\frac98 p^6u^3+\frac52 p^6u^2x^2-\frac{21}{4}p^6u^2x+\frac94p^6u^2
+4p^6ux^3+\frac92p^6ux^2+p^6ux\\[2mm]
&\quad-\frac32p^6u +2p^6x^4-3p^6x^3+p^6x^2-3p^6x+p^6
+9p^5u^2x^2-9p^5u^2\\[2mm]
&\quad-4p^5ux^3-6p^5ux^2+4p^5ux+6p^5u -4p^5x^4+8p^5x^3+12p^5x^2\\[2mm]
&\quad-8p^5x-8p^5
-14p^4u^2x^2+\frac{21}{4}p^4u^2x-8p^4ux^3-12p^4ux^2\\[2mm]
&\quad-p^4ux+9p^4u -4p^4x^4-21p^4x^3+8p^4x^2+21p^4x-18p^4
-9p^3u^2x^2\\[2mm]
&\quad+9p^3u^2-6p^3ux^3+6p^3ux^2+6p^3ux-6p^3u +8p^3x^4+16p^3x^3\\[2mm]
&\quad-16p^3x^2-16p^3x+8p^3
+\frac{23}{2}p^2u^2x^2+13p^2ux^3+\frac{15}{2}p^2ux^2-9p^2u\\[2mm]
&\quad +2p^2x^4+33p^2x^3-9p^2x^2-36p^2x+18p^2
+10pux^3-10pux-4px^4\\[2mm]
&\quad-24px^3+4px^2+24px-9ux^3-9x^3+18x.
\end{aligned}
\]
We first expand \(\widetilde R_2\) in the tensor-product Bernstein basis of degree $(6, 4, 3)$:
\[
\widetilde R_2(p,x,u)=\sum_{i=0}^{6}\sum_{j=0}^{4}\sum_{k=0}^{3}
\beta_{i,j,k}\,B_i^6(p)\,B_j^4(x)\,B_k^3(u),
\]
where
\[
B_r^N(s)=\binom Nr s^r(1-s)^{N-r}.
\]
Write
\[
M_k=(\beta_{i,j,k})_{0\le i\le 6,\ 0\le j\le 4},
\qquad k=0,1,2,3.
\]
The rows correspond to \(i=0,1,\dots,6\), and the columns correspond to \(j=0,1,\dots,4\).\\[2mm]
For \(k=0\),
\[
M_0=
\begin{pmatrix}
0 & \frac92 & 9 & \frac{45}{4} & 9 \\[5mm]
0 & \frac{11}{2} & \frac{100}{9} & \frac{163}{12} & 9 \\[5mm]
\frac65 & \frac{71}{10} & \frac{1181}{90} & \frac{467}{30} & \frac{143}{15} \\[5mm]
4 & \frac{19}{2} & \frac{149}{10} & \frac{84}{5} & \frac{53}{5} \\[5mm]
\frac{38}{5} & \frac{241}{20} & \frac{159}{10} & \frac{333}{20} & \frac{169}{15} \\[5mm]
\frac{26}{3} & \frac{139}{12} & \frac{27}{2} & \frac{53}{4} & \frac{29}{3} \\[5mm]
1 & 1 & 1 & 1 & 1
\end{pmatrix}.
\]
For \(k=1\),
\[
M_1=
\begin{pmatrix}
0 & \frac92 & 9 & \frac{21}{2} & 6 \\[5mm]
0 & \frac{193}{36} & \frac{65}{6} & \frac{113}{9} & 6 \\[5mm]
1 & \frac{298}{45} & \frac{2231}{180} & \frac{853}{60} & \frac{611}{90} \\[5mm]
\frac{33}{10} & \frac{1009}{120} & \frac{811}{60} & \frac{181}{12} & \frac{251}{30} \\[5mm]
\frac{31}{5} & \frac{917}{90} & \frac{413}{30} & \frac{523}{36} & \frac{143}{15} \\[5mm]
7 & \frac{19}{2} & \frac{67}{6} & \frac{395}{36} & \frac{71}{9} \\[5mm]
\frac12 & \frac12 & \frac12 & \frac12 & \frac12
\end{pmatrix}.
\]

For \(k=2\),
\[
M_2=
\begin{pmatrix}
0 & \frac92 & 9 & \frac{39}{4} & 3 \\[5mm]
0 & \frac{47}{9} & \frac{95}{9} & \frac{415}{36} & 3 \\[5mm]
\frac45 & \frac{553}{90} & \frac{6323}{540} & \frac{2339}{180} & \frac{43}{10} \\[5mm]
\frac{11}{4} & \frac{112}{15} & \frac{4459}{360} & \frac{553}{40} & \frac{69}{10} \\[5mm]
\frac{27}{5} & \frac{6449}{720} & \frac{13387}{1080} & \frac{3217}{240} & \frac{329}{36} \\[5mm]
\frac{19}{3} & \frac{137}{16} & \frac{81}{8} & \frac{1459}{144} & \frac{277}{36} \\[5mm]
\frac34 & \frac34 & \frac34 & \frac34 & \frac34
\end{pmatrix}.
\]

For \(k=3\),
\[
M_3=
\begin{pmatrix}
0 & \frac92 & 9 & 9 & 0 \\[5mm]
0 & \frac{61}{12} & \frac{185}{18} & \frac{21}{2} & 0 \\[5mm]
\frac35 & \frac{17}{3} & \frac{166}{15} & \frac{119}{10} & \frac{31}{15} \\[5mm]
\frac{47}{20} & \frac{267}{40} & \frac{1381}{120} & \frac{521}{40} & \frac{31}{5} \\[5mm]
\frac{26}{5} & \frac{401}{48} & \frac{4243}{360} & \frac{3187}{240} & \frac{121}{12} \\[5mm]
\frac{20}{3} & \frac{421}{48} & \frac{83}{8} & \frac{515}{48} & \frac{109}{12} \\[5mm]
\frac{23}{8} & \frac{23}{8} & \frac{23}{8} & \frac{23}{8} & \frac{23}{8}
\end{pmatrix}.
\]
By the Bernstein enclosure property,
\[
\widetilde R_2(p,x,u)\le \max_{0\le i\le 6,\ 0\le j\le 4,\ 0\le k\le 3}\beta_{i,j,k}.
\]
From the matrices above, the largest coefficient is
\[
\beta_{3,3,0}=\frac{84}{5}.
\]
Hence
\[
\widetilde R_2(p,x,u)\le \frac{84}{5}
\qquad\text{for all }(p,x,u)\in[0, 1]\times[0, 1]\times[0, 1].
\]
Returning to \(u=2t\), we obtain
\[
R_2(p_1,x,t)\le \frac{84}{5}
\qquad\text{for all }0\le p_1\le 1,\ 0\le x\le 1,\ 0\le t\le 1/2.
\]
We now subdivide the \((p_1,x)\)-rectangle into four parts and apply the Bernstein method again on each part, keeping \(0\le t\le 1/2\) unchanged. We have the following four sub-regions of the region $\mathcal{R}.$ 
\begin{align*}
Q_{11} &= \left[0,\frac12\right]\times\left[0,\frac12\right]\times[0,1],
\qquad
Q_{12} = \left[0,\frac12\right]\times\left[\frac12,1\right]\times[0,1], \\[3mm]
Q_{21} &= \left[\frac12,1\right]\times\left[0,\frac12\right]\times[0,1],
\qquad
Q_{22} = \left[\frac12,1\right]\times\left[\frac12,1\right]\times[0,1].
\end{align*}
The corresponding affine changes are
\begin{align*}
Q_{11}&:\quad p=\frac P2,\quad x=\frac X2,
\\[2mm]
Q_{12}&:\quad p=\frac P2,\quad x=\frac{1+X}{2},
\\[2mm]
Q_{21}&:\quad p=\frac{1+P}{2},\quad x=\frac X2,
\\[2mm]
Q_{22}&:\quad p=\frac{1+P}{2},\quad x=\frac{1+X}{2}.
\end{align*}
For each subbox we write the transformed polynomial in the Bernstein basis of degree $(6, 4, 3)$,
\[
\widetilde R_2^{(\nu)}(P,X,U)
=
\sum_{i=0}^6\sum_{j=0}^4\sum_{k=0}^3
\beta^{(\nu)}_{i,j,k}\,B_i^6(P)\,B_j^4(X)\,B_k^3(U),
\]
where \(\nu\in\{11,12,21,22\}\).\\[2mm]
For \(Q_{11}\), the coefficient matrices \(M_k^{(11)}=(\beta^{(11)}_{i,j,k})\) are
\[
M_0^{(11)}=
\begin{pmatrix}
0 & \frac94 & \frac92 & \frac{207}{32} & \frac{63}{8}\\[5mm]
0 & \frac52 & \frac{361}{72} & \frac{691}{96} & \frac{139}{16}\\[5mm]
\frac{3}{10} & \frac{119}{40} & \frac{8167}{1440} & \frac{15361}{1920} & \frac{143}{15}\\[5mm]
\frac{19}{20} & \frac{297}{80} & \frac{1039}{160} & \frac{5683}{640} & \frac{3329}{320}\\[5mm]
\frac{77}{40} & \frac{2999}{640} & \frac{4291}{576} & \frac{75247}{7680} & \frac{4321}{384}\\[5mm]
\frac{37}{12} & \frac{737}{128} & \frac{9685}{1152} & \frac{16325}{1536} & \frac{9185}{768}\\[5mm]
\frac{265}{64} & \frac{3401}{512} & \frac{4653}{512} & \frac{22745}{2048} & \frac{1571}{128}
\end{pmatrix},
\]

\[
M_1^{(11)}=
\begin{pmatrix}
0 & \frac94 & \frac92 & \frac{51}{8} & \frac{15}{2}\\[5mm]
0 & \frac{355}{144} & \frac{89}{18} & \frac{4037}{576} & \frac{197}{24}\\[5mm]
\frac14 & \frac{257}{90} & \frac{351}{64} & \frac{22099}{2880} & \frac{3211}{360}\\[5mm]
\frac{63}{80} & \frac{6619}{1920} & \frac{11773}{1920} & \frac{4283}{512} & \frac{18469}{1920}\\[5mm]
\frac{127}{80} & \frac{1349}{320} & \frac{823}{120} & \frac{69503}{7680} & \frac{19709}{1920}\\[5mm]
\frac{81}{32} & \frac{5815}{1152} & \frac{17413}{2304} & \frac{22175}{2304} & \frac{24775}{2304}\\[5mm]
\frac{433}{128} & \frac{731}{128} & \frac{8199}{1024} & \frac{20231}{2048} & \frac{5571}{512}
\end{pmatrix},
\]
\\[2mm]
\[
M_2^{(11)}=
\begin{pmatrix}
0 & \frac94 & \frac92 & \frac{201}{32} & \frac{57}{8}\\[5mm]
0 & \frac{175}{72} & \frac{39}{8} & \frac{491}{72} & \frac{371}{48}\\[5mm]
\frac15 & \frac{197}{72} & \frac{45791}{8640} & \frac{42359}{5760} & \frac{11983}{1440}\\[5mm]
\frac{103}{160} & \frac{3073}{960} & \frac{66767}{11520} & \frac{15151}{1920} & \frac{17083}{1920}\\[5mm]
\frac{53}{40} & \frac{29347}{7680} & \frac{219659}{34560} & \frac{9695}{1152} & \frac{13577}{1440}\\[5mm]
\frac{413}{192} & \frac{20801}{4608} & \frac{31817}{4608} & \frac{40793}{4608} & \frac{11323}{1152}\\[5mm]
\frac{747}{256} & \frac{10409}{2048} & \frac{3715}{512} & \frac{9241}{1024} & \frac{635}{64}
\end{pmatrix},
\]
\\[2mm]
\[
M_3^{(11)}=
\begin{pmatrix}
0 & \frac94 & \frac92 & \frac{99}{16} & \frac{27}{4}\\[5mm]
0 & \frac{115}{48} & \frac{173}{36} & \frac{1273}{192} & \frac{29}{4}\\[5mm]
\frac{3}{20} & \frac{157}{60} & \frac{737}{144} & \frac{6761}{960} & \frac{743}{96}\\[5mm]
\frac{83}{160} & \frac{1903}{640} & \frac{4213}{768} & \frac{19091}{2560} & \frac{659}{80}\\[5mm]
\frac{91}{80} & \frac{8967}{2560} & \frac{13691}{2304} & \frac{30319}{3840} & \frac{8393}{960}\\[5mm]
\frac{373}{192} & \frac{6385}{1536} & \frac{29713}{4608} & \frac{133}{16} & \frac{147}{16}\\[5mm]
\frac{1415}{512} & \frac{9779}{2048} & \frac{7017}{1024} & \frac{8767}{1024} & \frac{1205}{128}
\end{pmatrix}.
\]\\
It is easy to see that
\[
\max_{i,j,k}\beta^{(11)}_{i,j,k}=\frac{1571}{128},
\]
and this occurs in \(M_0^{(11)}\), last row and last column, that is, at
\[
\beta^{(11)}_{6,4,0}=\frac{1571}{128}.
\]
Hence
\[
\widetilde{R}_2^{(11)}(P,X,U)\le \frac{1571}{128},
\]
so equivalently
\[
R_2\le \frac{1571}{128}\approx 12.2734375
\qquad \text{on } Q_{11}.
\]
Therefore
\[
R_2(p,x,t)\le \max_{i,j,k}\beta^{(11)}_{i,j,k}
\qquad \text{for all }(p,x,t)\in Q_{11}.
\]
For \(Q_{12}\), the coefficient matrices \(M_k^{(12)}=(\beta^{(12)}_{i,j,k})\) are

\[
M_0^{(12)}=
\begin{pmatrix}
\frac{63}{8} & \frac{297}{32} & \frac{81}{8} & \frac{81}{8} & 9\\[5mm]
\frac{139}{16} & \frac{977}{96} & \frac{395}{36} & \frac{257}{24} & 9\\[5mm]
\frac{143}{15} & \frac{21247}{1920} & \frac{4249}{360} & \frac{5431}{480} & \frac{137}{15}\\[5mm]
\frac{3329}{320} & \frac{7633}{640} & \frac{1007}{80} & \frac{1907}{160} & \frac{47}{5}\\[5mm]
\frac{4321}{384} & \frac{32531}{2560} & \frac{76429}{5760} & \frac{3987}{320} & \frac{1169}{120}\\[5mm]
\frac{9185}{768} & \frac{6805}{512} & \frac{3955}{288} & \frac{1231}{96} & \frac{241}{24}\\[5mm]
\frac{1571}{128} & \frac{27527}{2048} & \frac{1761}{128} & \frac{3277}{256} & \frac{323}{32}
\end{pmatrix},
\]\\
\[
M_1^{(12)}=
\begin{pmatrix}
\frac{15}{2} & \frac{69}{8} & 9 & \frac{33}{4} & 6\\[5mm]
\frac{197}{24} & \frac{5419}{576} & \frac{1403}{144} & \frac{631}{72} & 6\\[5mm]
\frac{3211}{360} & \frac{3253}{320} & \frac{30151}{2880} & \frac{4477}{480} & \frac{2231}{360}\\[5mm]
\frac{18469}{1920} & \frac{83507}{7680} & \frac{5351}{480} & \frac{4759}{480} & \frac{791}{120}\\[5mm]
\frac{19709}{1920} & \frac{88169}{7680} & \frac{22501}{1920} & \frac{20077}{1920} & \frac{853}{120}\\[5mm]
\frac{24775}{2304} & \frac{9125}{768} & \frac{9271}{768} & \frac{693}{64} & \frac{547}{72}\\[5mm]
\frac{5571}{512} & \frac{24337}{2048} & \frac{12305}{1024} & \frac{1387}{128} & \frac{251}{32}
\end{pmatrix},
\]\\

\[
M_2^{(12)}=
\begin{pmatrix}
\frac{57}{8} & \frac{255}{32} & \frac{63}{8} & \frac{51}{8} & 3\\[5mm]
\frac{371}{48} & \frac{311}{36} & \frac{613}{72} & \frac{491}{72} & 3\\[5mm]
\frac{11983}{1440} & \frac{1189}{128} & \frac{79229}{8640} & \frac{591}{80} & \frac{133}{40}\\[5mm]
\frac{17083}{1920} & \frac{3803}{384} & \frac{22627}{2304} & \frac{15473}{1920} & \frac{159}{40}\\[5mm]
\frac{13577}{1440} & \frac{20047}{1920} & \frac{359651}{34560} & \frac{67217}{7680} & \frac{13957}{2880}\\[5mm]
\frac{11323}{1152} & \frac{16597}{1536} & \frac{49813}{4608} & \frac{4769}{512} & \frac{3301}{576}\\[5mm]
\frac{635}{64} & \frac{11079}{1024} & \frac{5553}{512} & \frac{19527}{2048} & \frac{815}{128}
\end{pmatrix},
\]\\

\[
M_3^{(12)}=
\begin{pmatrix}
\frac{27}{4} & \frac{117}{16} & \frac{27}{4} & \frac92 & 0\\[5mm]
\frac{29}{4} & \frac{1511}{192} & \frac{1049}{144} & \frac{39}{8} & 0\\[5mm]
\frac{743}{96} & \frac{8099}{960} & \frac{1423}{180} & \frac{1319}{240} & \frac{31}{60}\\[5mm]
\frac{659}{80} & \frac{4617}{512} & \frac{33047}{3840} & \frac{813}{128} & \frac{31}{20}\\[5mm]
\frac{8393}{960} & \frac{2455}{256} & \frac{107491}{11520} & \frac{3761}{512} & \frac{2837}{960}\\[5mm]
\frac{147}{16} & \frac{161}{16} & \frac{45841}{4608} & \frac{12721}{1536} & \frac{853}{192}\\[5mm]
\frac{1205}{128} & \frac{10513}{1024} & \frac{10509}{1024} & \frac{18257}{2048} & \frac{2909}{512}
\end{pmatrix}.
\]\\
Therefore,
\[
R_2(p,x,t)\le \max_{i,j,k}\beta^{(12)}_{i,j,k}
\qquad \text{for all }(p,x,t)\in Q_{12},
\]
thus, we have  
\[
\max_{Q_{12}} R_2 \le \max_{i,j,k}\beta^{(12)}_{i,j,k}=\frac{1761}{128}\approx 13.7578125.
\]
For \(Q_{21}\), the coefficient matrices \(M_k^{(21)}=(\beta^{(21)}_{i,j,k})\) are
\[
M_0^{(21)}=
\begin{pmatrix}
\frac{265}{64} & \frac{3401}{512} & \frac{4653}{512} & \frac{22745}{2048} & \frac{1571}{128}\\[5mm]
\frac{499}{96} & \frac{1927}{256} & \frac{22507}{2304} & \frac{35585}{3072} & \frac{9667}{768}\\[5mm]
\frac{1477}{240} & \frac{329}{40} & \frac{11719}{1152} & \frac{14987}{1280} & \frac{1601}{128}\\[5mm]
\frac{269}{40} & \frac{169}{20} & \frac{3207}{320} & \frac{3589}{320} & \frac{3771}{320}\\[5mm]
\frac{389}{60} & \frac{1243}{160} & \frac{1067}{120} & \frac{9313}{960} & \frac{803}{80}\\[5mm]
\frac{29}{6} & \frac{89}{16} & \frac{37}{6} & \frac{631}{96} & \frac{161}{24}\\[5mm]
1 & 1 & 1 & 1 & 1
\end{pmatrix},
\]\\
\[
M_1^{(21)}=
\begin{pmatrix}
\frac{433}{128} & \frac{731}{128} & \frac{8199}{1024} & \frac{20231}{2048} & \frac{5571}{512}\\[5mm]
\frac{271}{64} & \frac{7343}{1152} & \frac{38965}{4608} & \frac{93379}{9216} & \frac{6341}{576}\\[5mm]
\frac{799}{160} & \frac{19781}{2880} & \frac{99703}{11520} & \frac{7247}{720} & \frac{7759}{720}\\[5mm]
\frac{217}{40} & \frac{891}{128} & \frac{1609}{192} & \frac{72703}{7680} & \frac{3833}{384}\\[5mm]
\frac{207}{40} & \frac{907}{144} & \frac{10501}{1440} & \frac{46177}{5760} & \frac{11987}{1440}\\[5mm]
\frac{15}{4} & \frac{35}{8} & \frac{235}{48} & \frac{3023}{576} & \frac{773}{144}\\[5mm]
\frac12 & \frac12 & \frac12 & \frac12 & \frac12
\end{pmatrix},
\]\\
\[
M_2^{(21)}=
\begin{pmatrix}
\frac{747}{256} & \frac{10409}{2048} & \frac{3715}{512} & \frac{9241}{1024} & \frac{635}{64}\\[5mm]
\frac{1415}{384} & \frac{52079}{9216} & \frac{35053}{4608} & \frac{5297}{576} & \frac{11537}{1152}\\[5mm]
\frac{4217}{960} & \frac{70213}{11520} & \frac{268199}{34560} & \frac{6991}{768} & \frac{49}{5}\\[5mm]
\frac{387}{80} & \frac{3979}{640} & \frac{86741}{11520} & \frac{5483}{640} & \frac{17473}{1920}\\[5mm]
\frac{1129}{240} & \frac{32867}{5760} & \frac{57173}{8640} & \frac{2807}{384} & \frac{3671}{480}\\[5mm]
\frac{85}{24} & \frac{787}{192} & \frac{439}{96} & \frac{2827}{576} & \frac{727}{144}\\[5mm]
\frac34 & \frac34 & \frac34 & \frac34 & \frac34
\end{pmatrix},
\]\\

\[
M_3^{(21)}=
\begin{pmatrix}
\frac{1415}{512} & \frac{9779}{2048} & \frac{7017}{1024} & \frac{8767}{1024} & \frac{1205}{128}\\[5mm]
\frac{2753}{768} & \frac{16567}{3072} & \frac{1045}{144} & \frac{4511}{512} & \frac{617}{64}\\[5mm]
\frac{8489}{1920} & \frac{22943}{3840} & \frac{2903}{384} & \frac{1067}{120} & \frac{9263}{960}\\[5mm]
\frac{1633}{320} & \frac{407}{64} & \frac{5845}{768} & \frac{4433}{512} & \frac{297}{32}\\[5mm]
\frac{2569}{480} & \frac{12043}{1920} & \frac{4117}{576} & \frac{7547}{960} & \frac{3979}{480}\\[5mm]
\frac{229}{48} & \frac{339}{64} & \frac{553}{96} & \frac{587}{96} & \frac{101}{16}\\[5mm]
\frac{23}{8} & \frac{23}{8} & \frac{23}{8} & \frac{23}{8} & \frac{23}{8}
\end{pmatrix}.
\]\\
Therefore
\[
R_2(p,x,t)\le \max_{i,j,k}\beta^{(21)}_{i,j,k}
\qquad \text{for all }(p,x,t)\in Q_{21}.
\]
Thus, we have  
\[
\max_{Q_{21}} R_2 \le \frac{9667}{768}\approx 12.5872395833.
\]
For \(Q_{22}\), the coefficient matrices \(M_k^{(22)}=(\beta^{(22)}_{i,j,k})\) are

\[
M_0^{(22)}=
\begin{pmatrix}
\frac{1571}{128} & \frac{27527}{2048} & \frac{1761}{128} & \frac{3277}{256} & \frac{323}{32}\\[5mm]
\frac{9667}{768} & \frac{13917}{1024} & \frac{7939}{576} & \frac{4907}{384} & \frac{487}{48}\\[5mm]
\frac{1601}{128} & \frac{17033}{1280} & \frac{77009}{5760} & \frac{2969}{240} & \frac{199}{20}\\[5mm]
\frac{3771}{320} & \frac{3953}{320} & \frac{787}{64} & \frac{1819}{160} & \frac{93}{10}\\[5mm]
\frac{803}{80} & \frac{9959}{960} & \frac{819}{80} & \frac{303}{32} & \frac{79}{10}\\[5mm]
\frac{161}{24} & \frac{219}{32} & \frac{161}{24} & \frac{299}{48} & \frac{16}{3}\\[5mm]
1 & 1 & 1 & 1 & 1
\end{pmatrix},
\]\\
\[
M_1^{(22)}=
\begin{pmatrix}
\frac{5571}{512} & \frac{24337}{2048} & \frac{12305}{1024} & \frac{1387}{128} & \frac{251}{32}\\[5mm]
\frac{6341}{576} & \frac{36511}{3072} & \frac{18373}{1536} & \frac{347}{32} & \frac{1165}{144}\\[5mm]
\frac{7759}{720} & \frac{919}{80} & \frac{14719}{1280} & \frac{20137}{1920} & \frac{1457}{180}\\[5mm]
\frac{3833}{384} & \frac{80617}{7680} & \frac{20047}{1920} & \frac{9193}{960} & \frac{917}{120}\\[5mm]
\frac{11987}{1440} & \frac{16573}{1920} & \frac{767}{90} & \frac{3767}{480} & \frac{2323}{360}\\[5mm]
\frac{773}{144} & \frac{3161}{576} & \frac{43}{8} & \frac{715}{144} & \frac{151}{36}\\[5mm]
\frac12 & \frac12 & \frac12 & \frac12 & \frac12
\end{pmatrix},
\]\\
\[
M_2^{(22)}=
\begin{pmatrix}
\frac{635}{64} & \frac{11079}{1024} & \frac{5553}{512} & \frac{19527}{2048} & \frac{815}{128}\\[5mm]
\frac{11537}{1152} & \frac{65}{6} & \frac{50141}{4608} & \frac{9989}{1024} & \frac{2017}{288}\\[5mm]
\frac{49}{5} & \frac{40309}{3840} & \frac{364571}{34560} & \frac{36991}{3840} & \frac{21287}{2880}\\[5mm]
\frac{17473}{1920} & \frac{18497}{1920} & \frac{111317}{11520} & \frac{215}{24} & \frac{1163}{160}\\[5mm]
\frac{3671}{480} & \frac{5111}{640} & \frac{13771}{1728} & \frac{14327}{1920} & \frac{455}{72}\\[5mm]
\frac{727}{144} & \frac{2989}{576} & \frac{493}{96} & \frac{2783}{576} & \frac{38}{9}\\[5mm]
\frac34 & \frac34 & \frac34 & \frac34 & \frac34
\end{pmatrix},
\]
\\[2mm]
\[
M_3^{(22)}=
\begin{pmatrix}
\frac{1205}{128} & \frac{10513}{1024} & \frac{10509}{1024} & \frac{18257}{2048} & \frac{2909}{512}\\[5mm]
\frac{617}{64} & \frac{5361}{512} & \frac{12185}{1152} & \frac{29329}{3072} & \frac{5315}{768}\\[5mm]
\frac{9263}{960} & \frac{333}{32} & \frac{6777}{640} & \frac{7585}{768} & \frac{5063}{640}\\[5mm]
\frac{297}{32} & \frac{5071}{512} & \frac{7759}{768} & \frac{309}{32} & \frac{2663}{320}\\[5mm]
\frac{3979}{480} & \frac{8369}{960} & \frac{25517}{2880} & \frac{5499}{640} & \frac{249}{32}\\[5mm]
\frac{101}{16} & \frac{625}{96} & \frac{629}{96} & \frac{409}{64} & \frac{287}{48}\\[5mm]
\frac{23}{8} & \frac{23}{8} & \frac{23}{8} & \frac{23}{8} & \frac{23}{8}
\end{pmatrix}.
\]
Therefore
\[
\max_{Q_{22}} R_2(p, x, t) \le \max_{i,j,k}\beta^{(22)}_{i,j,k}=\frac{7939}{576}\approx 13.7829861111.
\]
Thus the Bernstein upper bounds on the four subregions are
\begin{align*}
Q_{11}:\ \frac{1571}{128},\qquad
Q_{12}:\ \frac{1761}{128},\qquad
Q_{21}:\ \frac{9667}{768},\qquad
Q_{22}:\ \frac{7939}{576}.
\end{align*}
Taking the largest of these four numbers, we conclude that
\[
R_2(p_1,x,t)\le \frac{7939}{576}
\qquad
\text{for all }
0\le p_1\le 1,\ 0\le x\le 1,\ 0\le t\le 1/2,
\]
precisely, 
$
{7939}/{576}\approx 13.7829861111.
$
Exhausting both the cases $y=0$ and $y=1$ for $H_1(p_1, x, y, t),$ we obtain
\[
144\, |H_{3}(1)|\le 16
\]
Hence, 
\begin{align*}
|H_{3}(1)|\le \frac{1}{9}.
\end{align*}
Now we prove that we obtain the extremal function for $H_{3}(1)$ by taking $w(z)=z^3,$ thus, we have 
\[
\varphi(z^3)=1+z^3+\frac{m}{n}z^6.
\]
Hence
\[
\frac{z f'(z)}{f(z)}=\varphi(z^3)=1+z^3+\frac{m}{n}z^6,
\]
so
\[
\frac{f'(z)}{f(z)}=\frac1z+z^2+\frac{m}{n}z^5.
\]
Integrating, we obtain
\[
\log f(z)=\log z+\frac{z^3}{3}+\frac{m}{6n}z^6+C.
\]
Therefore
\[
f(z)=C\,z\exp\left(\frac{z^3}{3}+\frac{m}{6n}z^6\right).
\]
Since \(f(0)=0\) and \(f'(0)=1\), we must have \(C=1\). Thus
\[
f(z)=z\exp\left(\frac{z^3}{3}+\frac{m}{6n}z^6\right).
\]
Set,
$
\alpha={m}/{6n}.
$
Using the exponential series,
\[
e^{\frac{z^3}{3}+\alpha z^6}
=1+\left(\frac{z^3}{3}+\alpha z^6\right)
+\frac12\left(\frac{z^3}{3}+\alpha z^6\right)^2
+\frac16\left(\frac{z^3}{3}\right)^3+\cdots.
\]
Now
\[
\frac12\left(\frac{z^3}{3}+\alpha z^6\right)^2
=\frac{z^6}{18}+\frac{\alpha}{3}z^9+\cdots,
\]
and
\[
\frac16\left(\frac{z^3}{3}\right)^3=\frac{z^9}{162}.
\]
Therefore
\[
f(z)=z+\frac{1}{3}z^4+\left(\alpha+\frac{1}{18}\right)z^7
+\left(\frac{\alpha}{3}+\frac{1}{162}\right)z^{10}+\cdots.
\]
Substituting \(\alpha={m}/{6n}\), we get
\[
f(z)=z+\frac{1}{3}z^4+\frac{n+3m}{18n}z^7+\frac{n+9m}{162n}z^{10}+\cdots.
\]
Thus, we have
\[
a_2=0,\qquad a_3=0,\qquad a_4=\frac13,\qquad a_5=0.
\]
Hence, 
\begin{align*}
|H_3(1)|=\frac19,
\end{align*}
which implies that $w(z)=z^3$ provides the extremal function, which proves the theorem. 
\end{pf}
\vspace{1mm}

\noindent\textbf{Compliance of Ethical Standards:}\\

\noindent\textbf{Conflict of interest.} The authors declare that there is no conflict  of interest regarding the publication of this paper.
\vspace{1.5mm}

\noindent\textbf{Data availability statement.}  Data sharing is not applicable to this article as no datasets were generated or analyzed during the current study.
\vspace{1.5mm}

\noindent\textbf{Authors contributions.} Both the authors have made equal contributions in reading, writing, and preparing the manuscript.\\

\noindent\textbf{Acknowledgment:} 
The second named author acknowledges financial support from the Council of Scientific and Industrial Research (CSIR), Government of India, through a CSIR Fellowship.

\end{document}